\documentclass[10pt]{article}
\usepackage{amssymb,amsmath,mathdots,amsthm}
\usepackage{epsfig,graphicx}
\usepackage{natbib}
\usepackage{enumitem}

\newtheorem{theorem}{Theorem}
\newtheorem{lemma}{Lemma}

\newcommand{\eq}[1]{(\ref{#1})}

\newcommand{\la}{\lambda}

\newcommand{\eps}{\varepsilon}

\newcommand{\RR}{\mathbb{R}}

\renewcommand{\Re}{\mathop{\rm Re}\nolimits}

\newcommand{\md}{d\kern-0.035cm\char39\kern-0.03cm}

\begin{document}

\title{Existence of traveling waves for the generalized FKPP equation}

\author{Richard Koll{\'a}r \\
Department of Applied Mathematics and Statistics\\
Faculty of Mathematics, Physics and Informatics \\
Comenius University \\
Mlynsk{\'a} dolina,  Bratislava, Slovakia \\
E-mail: {\tt kollar@fmph.uniba.sk}\\
\\
\\
Sebastian Novak \\
Institute of Science and Technology Austria \\
Am Campus 1, 3400 Klosterneuburg, Austria \\
e-mail: {\tt sebastian.novak@ist.ac.at}
}

\maketitle

\begin{abstract}
Variation in genotypes may be responsible for differences in dispersal rates, directional biases, 
and growth rates of individuals. These traits may favor certain genotypes and enhance their spatio-temporal 
spreading into areas occupied by the less advantageous genotypes. We study how these factors influence the 
speed of spreading in the case of two competing genotypes and show that under the assumption of maintenance 
of spatially homogeneous total population the dynamics  of the frequency of one of the genotypes is approximately 
described by the generalized Fisher-Kolmogorov-Petrovskii-Piskunov (FKPP) equation. This generalized FKPP equation 
with (nonlinear) frequency dependent  diffusion and advection terms admits traveling wave solutions (fronts/clines)  
that characterize the invasion of the dominant genotype. 
Our existence results generalize the classical theory for traveling waves for the FKPP with constant coefficients.
Moreover for the particular case of the quadratic (monostable) nonlinear growth-decay rate in the generalized FKPP 
we study in details the influence of the variance in diffusion and mean displacement rates of the two genotypes on 
the minimal wave propagation speed.  
\end{abstract}

\section{Introduction}
We study the one-dimensional {\it generalized Fisher-Kolmogorov-Petrovsky-Pis\-kunov  (gFKPP)} partial differential equation
\begin{equation}
\frac{\partial p}{\partial t} =  D(p) \frac{\partial^2 p}{\partial x^2} - M(p) \frac{\partial p}{\partial x} + f(p)\, ,
\label{FKPP}
\end{equation} 
where 
$$
D(p) = pD_2 + (1-p)D_1, \qquad \qquad M(p) = pM_2 + (1-p)M_1\, ,
$$
and $f(p)$ is a continuous function, $f(0) = f(1) = 0$, $D_1, D_2 > 0$ and $M_1, M_2$ are real parameters,  
$t$ is time, and $x\in \RR$ is the spatial coordinate.
The equation \eq{FKPP} is a generalization of the ubiquitous FKPP equation ($D_1 = D_2 = D$, $M_1 = M_2 =M=0$)
of the form 
\begin{equation}
\frac{\partial p}{\partial t} =   D \frac{\partial^2 p}{\partial x^2} + f(p)\, ,
\label{tFKPP}
\end{equation} 
that serves as a mathematical prototype of an interaction of diffusion with nonlinear reaction terms 
(\cite{Fisher, KPP}, see also \cite{Murray,Kot} for a survey of related literature, particularly 
in the context of mathematical ecology).  
The advection term $-M\partial p/\partial x$ can be removed from \eq{tFKPP} by a change of 
the reference frame $t \rightarrow t + Mx$.
In the traditional setting $p$ in \eq{tFKPP} represents a non-dimensionalized population of 
single species in one dimensional environment that can serve as a simple approximation of the 
evolution in the real environment  under the assumptions of directional homogeneity. 
The nonlinear reaction term has often the form that enables logistic growth of the population 
$f(p) = kp(1-p)$, and thus the population has two spatially homogeneous equilibria $p = 0$ 
(an elimination of the species) and $p = 1$ (a population saturation limited by the environmental resources), 
although various other nonlinearities with $f(0) = f(1) = 0$ are often considered. 

{\bf Generalized FKPP Equation.}
The FKPP equation in its non-dim\-en\-sio\-na\-lized form \eq{tFKPP} is typically formulated 
in the mathematical literature as a phenomenological model for the evolution of a population with 
an environmental saturation limit.
In evolutionary genetics, the FKPP equation can be derived as a model for the evolution of 
the frequency $p=p(x,t)$ of one of the two (competing) genotypes present in a spatially distributed 
population \citep{Nagylaki1975}.
Then, $D$ can be interpreted as the dispersal coefficient capturing the propensity and typical length 
scale of individual migration in space.
The linear advection term $M$ describes a directional bias in movement behavior and may be due to 
a systematic directional preference of individuals, or the presence of a gradient (slope, wind, current of water) in the habitat.
In the presence of type-dependent dispersal, i.e., the two genotypes have different dispersal coefficients 
$D_1$, $D_2$, and different advection terms $M_1$, $M_2$, the generalized FKPP equation \eq{FKPP} 
can be derived analogously to the type-independent case \citep{Nagylaki1980, Novak2014}.
In Section~\ref{s:derivation}, we present a formal argument well-known in the field of evolutionary 
genetics, but very rarely cited in the mathematical literature.
The nonlinear equation \eq{FKPP} is derived from the system of two uncoupled linear reaction-diffusion equations for two genotypes.
The coupling is introduced by the formal, and in the population genetics literature widely accepted,
assumption that the total population of the two genotypes is kept homogeneous in space.
This results in locally heterogeneous diffusion and advection terms $D(p)$ and $M(p)$, and a nonlinear term $f(p)$.
The growth rate $k$ in $f(p)=kp(1-p)$ is equal to the difference of the (absolute) growth rates $r_1$ 
and $r_2$ of the two genotypes, $k=r_1-r_2$, that may describe, e.g., the action of natural selection.
Thus, equation \eq{FKPP} can be used to study the influence of differences in dispersal coefficient, 
directional bias, and growth rates on the evolutionary dynamics of the genotypes.

{\bf Traveling Waves.}
In the context of the FKPP equation, traveling wave solutions
$$
p(x,t) = P(\xi), \qquad 
\xi = x -ct\, ,
$$
of \eq{tFKPP} (and alternatively of \eq{FKPP}), where $c$ is the traveling wave speed, are traditional 
tools to study spatial patterns of genotype frequencies.
\cite{Fisher} was the first to use traveling waves to model the invasion of an advantageous mutation into an ancestral population.
Modifying the function $f(p)$, traveling waves may also provide a theoretical mechanism to create a genetic barrier within a population leading to speciation \citep{Bazykin1969} or describe gene frequency patterns that emerge in hybrid zones after the secondary contact of species \citep{Barton1979}.
Similar patterns emerge in heterogeneous environments, i.e., if $f=f(p,x)$ depends explicitly on space \citep{Nagylaki1975}.

The traveling wave profile (also called a front or a cline) satisfies the second order ordinary differential equation
\begin{equation}
-cP' = D(P)P'' - M(P)P' + f(P)\, ,
\label{secondorder}
\end{equation}
where $P'$ denotes $dP/d \xi$. We will require that the admissible solutions of \eq{secondorder} 
on $\xi \in (-\infty, \infty)$ satisfy either 
\begin{equation}
\mbox{$P(\xi) \in [0,1]$ for all $\xi$ real,}\ \ 
\mbox{$P(\xi) \rightarrow 1$ as $\xi \rightarrow -\infty$,} \quad
\mbox{$P(\xi) \rightarrow 0$ as $\xi \rightarrow \infty$.}
\label{admsol2}
\end{equation}
or 
\begin{equation}
\mbox{$P(\xi) \in [0,1]$ for all $\xi$ real,}\ \ 
\mbox{$P(\xi) \rightarrow 0$ as $\xi \rightarrow -\infty$,} \quad
\mbox{$P(\xi) \rightarrow 1$ as $\xi \rightarrow \infty$.}
\label{admsol1}
\end{equation}
These solutions represent a spatio-temporal invasion of the dominant genotype into a region populated by its receding counterpart.

{\bf Literature.}
Existence of traveling waves and their stability for \eq{FKPP} with quadratic  and cubic $f(p)$ is a well studied subject.%
\footnote{The quadratic $f(p)$ is also called {\it monostable} as the spatially homogeneous reduced dynamical system 
$p_t = f(p)$ has in that case one stable and one unstable equilibrium. On the other hand,  a cubic $f(p)$ is called {\it bistable} as in that case there are two stable equilibria.}
\cite{Fisher} proposed the model and numerically calculated the wave profile for a quadratic $f(p)$ in his study of a propagation of an advantageous gene in a population. \cite{KPP} rephrased the problem in terms of dynamical systems, related existence of the traveling waves to existence of heteroclinic orbits  and showed that in the monostable case the central role is played by the so called critical speed that is the minimal speed for which the traveling wave of the type \eq{admsol2} exists.  

 The critical wave speed for monostable nonlinearities and the unique wave speed for bistable nonlinearities were characterized by minimax and maximin principles in \cite{Hadeler1975, Hadeler1987}. In \cite{Hadeler1987}  the ideas of \cite{Conley1978} (see also \cite{Smoller1982}) were reformulated and used to  characterize existence of traveling waves for \eq{FKPP} with a general nonlinearity. The argument is based on  a mechanical analogue of the system in which the wave speed $c$ plays the role of a (positive or negative) friction coefficient. Existence of traveling waves for general nonlinearities $f(p)$
is also discussed in details in  \cite{VVV94}  using variational principles. 

Stability of traveling waves in exponentially weighted Banach spaces was studied by  \cite{Sattinger1976}. More recently, existence of traveling waves for the degenerate parabolic equations of type 
$p_t = [D(p)p_x]_x + f(p)$,  where $D(0) = 0$ and $D(p)$ is strictly increasing, and $f(p) > 0$ for $p \in (0,1)$, was studied using shooting arguments in \cite{Maini1996} where authors also survey literature on the subject. 

The speed of the traveling wave for the monostable nonlinearity $f(p) =  kp(1-p)$ is determined by the instability of the homogeneous state $p = 0$ and the traveling wave is called pulled, as the leading edge of the wave at $p \approx 0$ pulls the bulk of the wave at $p \approx 1$. However, in applications the quadratic nonlinearity  does not accurately describe the growth of the population close to $p=0$. \cite{BD1997} demonstrated that the critical wave speed changes if the nonlinearity $f(p) =kp(1-p)$ is modified close to $p=0$ and \cite{DKP2007} and \cite{DK2015} showed how the speed of the wave is asymptotically modified if the nonlinearity $f(p)$ is altered in the $\eps$ neighborhood of $p=0$  using the geometric blow-up technique. The same phenomenon from a different perspective was analyzed in \cite{DMS2003} where a stochastic FKPP equation was considered. 

A more general problem of dynamics and asymptotic behavior of the solutions as $t \rightarrow \infty$ to the Cauchy problem for \eq{tFKPP} on $\RR^n$  was studied in \cite{AW1978} for a bistable $f(p)$. The authors show that for a certain class of initial data, close enough to a traveling wave profile,  the solution to \eq{tFKPP} asymptotically approaches the traveling wave solution. These results were extended by  \cite{FMcL1977,FMcL1980} who analyzed the problem in one dimension using the results on asymptotic stability of traveling waves. 
Very recently these results were extensively generalized using the phase plane analysis in the seminal works of  \cite{Polacik2015, Polacik2016} who was inspired by \cite{DGM2014}. His method  requires only Lipschitz continuity of $f(p)$ with multiple zeros in $[0,1]$. The techniques used by Pol{\'a}{\v c}ik are geometrical and they are not based on the stability of the traveling waves. Thus his results also extend to degenerate problems where $f'(p) = 0$ at some zeros of $f(p)=0$ for which the stability results are not, in general, available. Furthermore, he was able to remove the technical assumption on monotonicity of the initial data that was used in the existing literature. See \cite{Polacik2016} for more detailed list of the references on the subject. 

From the perspective of applications, the FKPP equation has a long tradition in modeling spatially distributed systems in many scientific disciplines. It has been applied in population genetics to model the dynamics of gene frequencies \citep{Fisher} to predict rates of introgression of genotypes, and how their spatial spread may be initialized and interrupted \citep{Barton2011}.
The FKPP equation also has ecological \citep{Matsushita1999} and chemical applications \citep{Xin2000}, as well as applications in evolutionary game theory as a framework to select the {\it spatially dominant} equilibrium from a set of evolutionarily stable strategies 
\citep{Hofbauer1999}.

{\bf Our Work.}
 In Section~\ref{s:derivation} we present  for the sake of completeness the derivation of \eq{FKPP} in the context of evolutionary genetics. 
Section~\ref{s:ds} contains reformulation of the main problem in the language of dynamical systems, description of the symmetries of the system, and also an introduction of a notation and a terminology used. Existence and nonexistence of the traveling wave solutions of \eq{FKPP} satisfying \eq{admsol2} and \eq{admsol1} is characterized in Theorem~\ref{th:existence} in Section \ref{s:existence} that is a consequence of Lemmae 1--4. The theorem characterizes  the type of the range of values of the wave speed $c$ for which the traveling wave exists depending on the number of roots of the nonlinearity $f(p)$ in $[0,1]$. Our results agree with the results for the FKPP equation \eq{tFKPP} \citep{Sattinger1976,Hadeler1987,VVV94}. 

Furthermore, in Section \ref{s:homogeneous}
in the particular case when the diffusion coefficients of both species agree ($D_1 = D_2= D$) and $f(p) = kp(1-p)$ we show that the range of speeds of traveling waves satisfying \eq{admsol2} (the results for \eq{admsol1} are analogous) is an interval $[c^{\ast}, \infty)$ for a specific value of $c^{\ast} = c^{\ast}(k, M_1, M_2, D)$. 
The results summarized in Theorem~\ref{th:homo} identify the role of advection terms $M_1$ and $M_2$ play in determining the traveling wave velocity. It is well-know that in the case of \eq{tFKPP} the critical lower bound $c^{\ast} = 2\sqrt{Df'(0)}$ corresponds to the natural threshold determined by the local dynamics of \eq{secondorder} close to $P=0$ and that the critical wave  for $c = c^{\ast}$ is pulled, i.e., the instability of the state $P=0$ pulls the wave forward. 
On the other hand, in the case of \eq{FKPP} the variable advection speed, $M_1 \neq M_2$, plays a significant role. First, if $M_2 - M_1 \le 2\sqrt{Df'(0)}$ then the drift $M_2$ does not influence $c^{\ast}$ and the critical wave is pulled with the speed $c^{\ast} = 2\sqrt{Df'(0)}$ in the reference frame moving with the velocity $M_1$. However, if $M_2 - M_1 > 2\sqrt{Df'(0)}$ then the drift
of the bulk of the wave at $P \approx 1$ is supercritical, i.e., it is faster than the pulling speed of the tail of the wave, $P \approx 0$, and the wave becomes pushed.%
\footnote{See \cite{Stokes1976, vanSaarloos2003, DKP2007} for more explanation of the term pushed and pulled wave in the context of front propagation in reaction-diffusion equations.}
 For $M_2 - M_1 \rightarrow \infty$ the velocity approaches the naturally expected value $(M_1+M_2)/2$ (in the static frame of reference). Our method of proof can be interpreted as a generalization of the ideas of \cite{Hadeler1987} although  it is formulated in the language of phase portrait analysis of a planar dynamical system rather than its mechanical analogue. 

In Section~\ref{s:nonhom} we discuss our numerical results in the case of non-uniform diffusion, $D_1 \neq D_2$. Although the pattern of dependence of $c^{\ast}$ on $M_2 - M_1$ remains the same, i.e., for $M_2 - M_1$ less than some transition value the minimum wave speed is equal to $2\sqrt{D_1f'(0)}$ and it corresponds to the pulled wave, beyond this transition value only pushed waves exists. We numerically calculate the value of $M_2 - M_1$ at which the transition occurs and the results are quite surprising. For moderate values of $D_2/D_1$ the transition point depends approximately linearly on $D_2$ (for a fixed value of $D_1$) and it moves to higher values for $D_2 < D_1$ and lower values for $D_2 > D_1$. However, if $D_2 \gg D_1$ or $D_2 \ll D_1$, the change of the location of the transition point turns the other way, even beyond the transition point for $D_1 = D_2$.  
On the other hand, the critical wave speed $c^{\ast}$ for large values of $M_2 - M_1$ grows approximately linear with $M_2 - M_1$. Our numerical simulations indicate that the asymptotic slope depends approximately linearly on the logarithm of $D_2/D_1$.  
Finally, in Section~\ref{s:conclusion} we discuss our results and formulate open problems stemming from our analysis.

\section{Derivation of the gFKPP Equation}
\label{s:derivation}
Within this section we derive the equation \eq{FKPP} using the steps in the formal argument  of \cite{Nagylaki1980} (see also \citep{Novak2014}). 
We show that under certain  specific assumptions the reduced dynamics characterized by \eq{gFKPP} is an approximation of the system of  reaction-diffusion equations describing the evolution of populations of two genotypes, where $p$ represents the fraction of one of the genotypes in the total population.

Consider the evolution of populations of $m$ different genotypes in a homogeneous one dimensional space. Each of the genotypes is characterized by its own genotype-specific dispersal rate (diffusion coefficient) $D_j$, mean displacement coefficient (advection, drift)  $M_j$, and growth rate $r_j$. 
The population dynamics is characterized by the system of reaction-diffusion equations
\begin{equation}
\partial_t n_j  =  D_j \partial_{xx} n_j - M_j \partial_x n_j + f_j (\vec{n})\, ,\quad  j = 1, 2, \dots, m, 
\label{eqn1}
\end{equation}
where $n_j = n_j(x,t)$ are the populations of individual genotypes, $D_j$, $M_j$, and $f_j$, respectively, their dispersal, mean displacement, and growth rates, and $\vec{n} = (n_1, n_2, \dots, n_m)$,
The variance in diffusion, advection, and growth coefficients of different genotypes is biologically justified. 
\cite{Edelaar2012} pointed out that dispersal properties often differ between the (geno-)types represented in natural populations. For instance, this is the case in aquatic species with differential capability of resisting a unidirectional current.
Also, \cite{Lutscher2007} used reaction-diffusion equations \eq{eqn1} to show conditions for this form of type-dependent dispersal under which inferior competitors may evade into upstream regions.

The total population of the individuals of all genotypes $N = n_1 + \dots + n_m$ satisfies the equation
\begin{equation}
\partial_t N  = \sum_{j=1}^{m} D_j \partial_{xx}  n_j   -  \sum_{j=1}^m M_j \partial_x n_j + \sum_{j=1}^m f_j (\vec{n})\, .
\label{totaleq}
\end{equation}
We denote $p_j = n_j / N$ the frequency of the $j$-th genotype in the population, i.e., $n_j = p_jN$, and 
$\vec{p}N = (p_1N, \dots, p_mN)$. 
Then \eq{totaleq} can be written as
\begin{equation}
\partial_t N  = \sum_{j=1}^{m} D_j \partial_{xx}  p_j N   -  \sum_{j=1}^m M_j \partial_x p_j N 
+ \sum_{j=1}^m f_j (\vec{p}N)\, .
\label{totaleqp}
\end{equation}
The dynamics of $p_j= n_j/N$ is governed by
\begin{eqnarray}
\partial_t p_j & = & \frac{1}{N} \left( \partial_t n_j - p_j \partial_t N\right) \nonumber \\
& = & 
\frac{1}{N}\Big[  D_j \partial_{xx}p_jN - M_j \partial_x p_j N+ f_j (\vec{p}N)  \Big. 
 \label{fullp} \\
& & \phantom{\frac{1}{N}aaa} 
\left. - 
p_j  \left(  \sum_{i=1}^{m} D_i \partial_{xx}  p_i N   -  \sum_{i=1}^m M_i \partial_x p_i N + \sum_{i=1}^m f_i (\vec{p}N)
\right)  \right]\, . \nonumber
\end{eqnarray}
At this point we make a formal assumption that $N$ is spatially homogeneous, i.e., $N = N(t)$; its validity is discussed below. Then the system \eq{fullp} reduces to 
\begin{eqnarray}
\partial_t p_j  & =& 
 D_j \partial_{xx}p_j - M_j \partial_x p_j + \frac{1}{N}f_j (\vec{p}N)\nonumber  \\ 
& &  - 
p_j  \left(  \sum_{i=1}^{m} D_i \partial_{xx}  p_i    -  \sum_{i=1}^m M_i \partial_x p_i  + \frac{1}{N}\sum_{i=1}^m f_i (\vec{p}N)
\right)\, . \label{redp}
\end{eqnarray}

The system \eq{redp} for $j = 1, \dots, m$, can be considered separately from the original problem with $N(t)$ as a time-dependent parameter. Denote $p_{total} = \sum_{j=1}^m p_j$ and sum \eq{redp} over all $j$ to obtain
\begin{equation}
\partial_t p_{total} = (1-p_{total}) \left(  \sum_{i=1}^{m} D_i \partial_{xx}  p_i    -  \sum_{i=1}^m M_i \partial_x p_i  + \frac{1}{N}\sum_{i=1}^m f_i (\vec{p}N)
\right)\, .
\label{redeq}
\end{equation}
Hence if $p_{total} = 1$ initially for $t = 0$, the sum of $p_j$ remains constant for all $t$. 
Moreover, it is easy to see that the flow \eq{redp} under the assumption $f_j(\vec{p}N) = 0$ for $p_j = 0$ preserves nonnegativity of all $p_j$, and thus we will refer to  $p_j$ as  frequencies. 

Next we consider a special case $m=2$. Denote $p = p_1$, then set $p_2 = 1-p$ by assuming $p_1 + p_2 = 1$ initially. Then the system \eq{redp} reduces to a single equation
\begin{equation}
\partial_t p = D(p) \partial_{xx} p  - M(p) \partial_x p + f(p,N) \, ,
\label{gFKPP}
\end{equation}
where 
\begin{gather*}
D(p) =  (1-p)D_1 + p D_2, \qquad
M(p) = (1-p)M_1 + pM_2, \\
f(p,N) =  (1-p)\frac{f_1}N - p\,\frac{f_2}{N}\, .
\end{gather*}

In the field of evolutionary genetics it is traditional to consider the linear growth rate of all genotypes but for illustrative purposes we also discuss here some alternative choices of growth functions $f_1$ and $f_2$. 

First, we assume that 
$$
f_1 = r_1n_1 = r_1p N, \qquad 
f_2 = r_2 n_2 = r_2 (1-p)N\, .
$$
Then
$$
f(p) = \frac{1}{N}  \left((1-p)r_1pN - pr_2(1-p) N\right) = (r_1 - r_2) p (1-p)\, .
$$
Therefore we recover the typical quadratic nonlinearity in the FKPP equation. Particularly note that the equation \eq{gFKPP} is  in this case independent of $N$. 
Also note that the nonlinearity originates in the different magnitude of the growth rates of the two genotypes, 
i.e., the equation \eq{gFKPP} can be used to study the effects of different dispersal, mean-displacement and linear growth 
rates on the frequencies of individual genotypes, although some caution is needed as the system \eq{redp} was derived from \eq{eqn1}
using the assumption on spatial homogeneity of $N$ that may not be, in general, completely satisfied. 

Another interesting case is the independent logistic growth of each genotype
$$
f_i = r_i n_i \left(1 - \frac{n_i}{K_i}\right)\, , \qquad i = 1,2,
$$
where $K_i = K_i(t)$ is the carrying capacity of the genotype $i$. In that case 
\begin{equation}
f(p) = p(1-p) \left[ (r_1 - r_2) - \left(\frac{r_1}{\alpha_1} \, p - \frac{r_2}{\alpha_2} (1-p)\right) \right]\, , 
\label{Fpcub}
\end{equation}
where $\alpha_i = K_i/N$. If $\alpha_1$ and $\alpha_2$ are constant and $\alpha_1, \alpha_2  < 1$, the expression 
on the right-hand side of \eq{Fpcub} can be written as $kp(1-p)(p-p_0)$ with $p_0 \in (0,1)$ and it corresponds to the Allee effect \citep{Murray}. In such a case $k > r_1 + r_2$. 

On the other hand, additive terms in $f_i$ in the form $n_ig(N)$ do not influence $f(p)$ as 
$$
\frac{1}{N} \left[ n_1 g(N) - p (n_1 g(N) + n_2 g(N))\right]  = 0\, .
$$

Finally, if the terms $f_i$ represent a direct competition between the genotypes $f_1 = -f_2 = g(n_1, n_2)$ then
$$
f(p,N) = \frac{1}{N} \left[ g + p(g-g)\right] = \frac{g}{N}\, .
$$
Particularly, if $g(n_1,n_2) = \gamma n_1 n_2$ 
$$
f(p,N) = \gamma p(1-p)N\, ,
$$
and thus in this case the equation \eq{gFKPP} is directly dependent on $N$. 

\vspace{\baselineskip}

\noindent
{\bf Consistency.} Without the assumption on spatial homogeneity of $N$ the equation \eq{fullp} has the form
\begin{eqnarray}
\partial_t p_j & = & 
 D_j \partial_{xx}p_j - M_j \partial_x p_j + \frac{1}{N}f_j (\vec{p}N) 
+ p_j \left(D_j  - \sum_{i=1}^m D_i p_i\right) \frac{\partial_{xx} N}{N} \nonumber \\
& &  - p_j  \left(  \sum_{i=1}^{m} D_i \partial_{xx}  p_i    -  \sum_{i=1}^m M_i \partial_x p_i  + \frac{1}{N}\sum_{i=1}^m f_i (\vec{p}N)
\right)  \nonumber \\
& & 
+ \left( 2D_j \partial_x p_j -  p_j \sum_{i=1}^m 2D_i \partial_x p_i  -  M_jp_j   + p_j \sum_{i=1}^m M_i p_i \right) \frac{\partial_{x} N}{N}\, .
\label{fullexp}
\end{eqnarray}
If the relative spatial variation of the total population $N$ starts and remains relatively small compared to the relative spatial variation of the frequencies $p_j$, 
$$\frac{|\partial_x N|}{N} \ll \frac{|\partial_x p_j|}{p_j}, \qquad
\frac{|\partial_{xx} N|}{N} \ll \frac{|\partial_{xx} p_j|}{p_j}, \quad 
\frac{|\partial_x N|}{N} \ll \max\left\{\frac{|\partial_{xx} p_j|}{|\partial_x p_j|}, \frac{M_j}{D_j} \right\},
$$
the correction terms  in \eq{fullexp} can be neglected compared to the terms in \eq{redp}.

Furthermore, if $D_i = D$ and $M_i = M$ for all $i = 1, \dots, m$, then under the assumption $\sum_{i=1}^{m}  p_i = 1$  for all  $x \in \RR$ initially, the terms $D - \sum_{i=1}^{m} D p_i$ and $-Mp_j + p_j \sum_{i=1}^m Mp_i$ 
in the second line of \eq{fullexp} vanish. Also, $\sum_{i=1}^m 2 D\partial_x p_i = 0$. Therefore   \eq{fullexp} reduces in that case to 
\begin{eqnarray}
\partial_t p_j & = & 
 D \partial_{xx}p_j - M \partial_x p_j + \left[ \frac{f_j (\vec{p}N)}{N}  - 
\frac{p_j}{N} \sum_{i=1}^m f_i (\vec{p}N)\right] + 2D\partial_x p_j \frac{\partial_x N}{N}. \ \ \ \ 
\label{fullexpred}
\end{eqnarray}
On the other hand,  the evolution of $N$ is  governed by the reduced equation \eq{totaleqp}:
\begin{equation}
\partial_t N  =  D \,\partial_{xx}  N   -  M \partial_x N 
+ \sum_{j=1}^m f_j (\vec{p}N)\, .
\label{totaleqpred}
\end{equation}
The coupled system \eq{fullexpred}-- \eq{totaleqpred} then characterizes the dynamics of $n_i = p_iN$ exactly for all $i = 1, \dots, m$.  This system also preserves the sum of $p_j$ equal to one and non-negativity of $p_j$, along with non-negativity of $N$, if $f_j(\vec{p}N) = 0$ for $N=0$.

\section{Dynamical System Reformulation}
\label{s:ds}
Problem \eq{secondorder} can be rewritten as the first order system
\begin{eqnarray}
P' & = & Q, \label{e1}\\
Q' & = & \frac{M(P)-c}{D(P)}\,Q - \frac{f(P)}{D(P)}\, .
\label{e2}
\end{eqnarray}

The fixed points of the two-dimensional dynamical system \eq{e1}--\eq{e2} are given by 
$(P^{\ast},0)$ where $f(P^{\ast})  = 0$, particularly $(0,0)$ and $(1,0)$ are equilibria. The traveling front solutions 
correspond to {\it admissible heteroclinic orbits} of \eq{e1}--\eq{e2} connecting the equilibria $(0,0)$ and $(1,0)$, i.e. solutions 
$(P(\xi),Q(\xi))$ of \eq{e1}--\eq{e2}
satisfying $0 <  P(\xi)  < 1$ for all $\xi$ real and one of the following conditions
\begin{eqnarray}
& &  \lim_{\xi \rightarrow -\infty} (P(\xi),Q(\xi)) =  (1,0) \qquad \mbox{and} \qquad \lim_{\xi \rightarrow \infty}(P(\xi),Q(\xi)) = (0,0),
\label{TypeA} \\
 & & \lim_{\xi \rightarrow -\infty} (P(\xi),Q(\xi)) =  (0,0) \qquad \mbox{and} \qquad \lim_{\xi \rightarrow \infty}(P(\xi),Q(\xi)) = (1,0).
\label{TypeB}
\end{eqnarray}

\noindent
{\bf Symmetries.}
There are two important symmetries of \eq{secondorder}.
The change of variables 
\begin{equation}
\widehat{P} = 1-P, \ \ 
\widehat{f}(P) = - f(1-P), \ \ 
\widehat{M_j} = M_{3-j},\ \ \
\widehat{D_j} = D_{3-j}, \ \ \
j = 1,2,
\label{sym1}
\end{equation}
transforms \eq{secondorder} to the same form with $P$ replaced by $\widehat{P}$ and $f$ by $\widehat{f}$. This transformation switches the heteroclinic orbits from $(0,0)$ to $(1,0)$ to orbits from $(1,0)$ to $(0,0)$ and vice-versa with the same velocity $c$ and $Q$ replaced by $-Q$. 
The equation \eq{secondorder} is also invariant with respect to the change of variables
\begin{equation}
\widehat{c} = - c, \qquad 
\widehat{\xi} = -\hat{\xi}, \qquad 
\widehat{M_1} = -M_1, \qquad 
\widehat{M_2} = -M_2.
\label{sym2}
\end{equation} 
This transformation changes both $Q$ to $-Q$ and $c$ to $-c$. 

\noindent
{\bf Linearization.}
The linearized flow of \eq{e1}--\eq{e2} at an equilibrium $(P,Q) = (P^{\ast},0)$ is given by the linear system 
$y' = Ay$, where $y = (P,Q)^T$, 
$$
A = A(P^{\ast}) = \left(
\begin{matrix}
0 & 1 \\
-\alpha & \beta 
\end{matrix}
\right),
$$
and 
$$
\alpha = \alpha(P^{\ast}) = \frac{f'(P^{\ast})}{D(P^{\ast})}, \quad
\beta = \beta(P^{\ast}) = \frac{M(P^{\ast})-c}{D(P^{\ast})}\, .
$$
The eigenvalues $\lambda^{\pm}_{P^{\ast}}$, $\Re \la^+_{P^{\ast}} \ge \Re \la^-_{P^{\ast}}$, are the roots of the 
characteristic quadratic equation
\begin{equation}
\lambda^2 - \beta \la + \alpha = 0\, . 
\label{quadratic}
\end{equation}
Therefore
\begin{eqnarray}
\lambda^{\pm}_{P^{\ast}} & =&  \frac{1}{2D(P^{\ast})} \left( M(P^{\ast}) - c \pm \sqrt{(M(P^{\ast}) - c )^2 - 4f'(P^{\ast})D(P^{\ast})}\right)
\, . 
\label{quadroot}
\end{eqnarray}
The eigenvectors corresponding to the eigenvalues $\la^{\pm}_{P^{\ast}}$ can be selected as  $(1,\la^{\pm}_{P^{\ast}})$.
 
\begin{figure}[t] 
\centering
\includegraphics[width=0.65\textwidth]{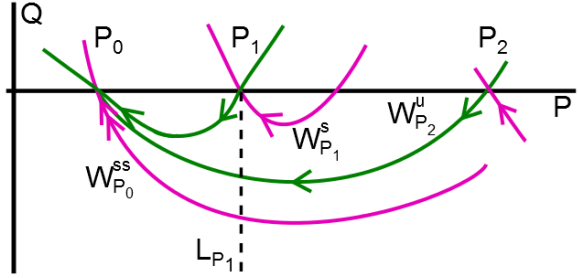}
\caption{Illustration of our notation and terminology. The points $P_0, P_1$ and $P_2$ are equilibria of the system, $P_1$ and $P_2$ are saddle points, $P_0$ is a stable node. $W^u_{P_2}$ is the unstable manifold of the saddle $P_2$, $W^s_{P_1}$ is the stable manifold of $P_1$, $W^{ss}_{P_0}$ is the fast stable manifold of $P_0$. All three approach the corresponding equilibria from the lower half plane $Q <0$. The (dashed) half-line $L_{P_1}$ originates at $(P_1, 0)$. The manifold $W^{ss}_{P_0}$ lies under $W_{P_2}^u$ and they both lie under $W_{P_1}^s$. The manifold $W^u_{P_2}$ overshoots $P_1$ while it does not reach equilibrium at $P=0$. 
\label{Fig:not}}
\end{figure}

\noindent
{\bf Notation.}
We will use the following notation and terminology, see Fig.~\ref{Fig:not}.
\begin{itemize}[leftmargin=\parindent]
\item
Since all the equilibria of the system \eq{e1}--\eq{e2} have the form $(P^{\ast},0)$, we will refer to it simply as $P^{\ast}$. 
\item
Due to the symmetries \eq{sym1} and \eq{sym2}  we can restrict our analysis without loss of generality to traveling waves satisfying \eq{TypeA}. Thus we can restrict our proofs solely to the lower half-plane of the phase plane $(P,Q)$.  
\item 
If $P^{\ast}$ is a saddle point of \eq{e1}--\eq{e2} we denote by $W^u_{P^{\ast}}$ the part of the unstable manifold of $P^{\ast}$ parametrized as $(P(\xi), Q(\xi))$ for which $P \rightarrow (P^{\ast})^-$ as $\xi \rightarrow -\infty$ (and $Q(\xi) \rightarrow 0^-$). Analogously we denote by $W^s_{P^{\ast}}$ the part of the stable manifold of $P^{\ast}$ for which $P \rightarrow (P^{\ast})^+$ as $\xi\rightarrow \infty$.
\item
If $P^{\ast}$ is a node such that the linearization of \eq{e1}--\eq{e2} at $(P,Q) = (P^{\ast}, 0)$ has two negative real eigenvalues, then we denote by $W^{ss}_{P^{\ast}}$ the part of the unique fast stable manifold for which $P \rightarrow (P^{\ast})^+$ as $\xi \rightarrow \infty$ and $P/Q \rightarrow \la_{P^{\ast}}^-$ as $\xi \rightarrow \infty$. 
\item
We denote by $L_{P^{\ast}}$ the half-line $\{(P,Q); P = P^{\ast}, Q < 0\}$. 
\item
We say that $W^u_{P_2}$ does not reach the equilibrium $P_1$, $P_1 < P_2$, if 
$W^u_{P_2}$ either does not intersect $L_{P_1}$ or if it intersects the segment $\{(P,Q); P \in (P_1, P_2), Q = 0\}$, for smaller value of $\xi$ than its first intersection with $L_{P_1}$.  
\item
On the other hand, we say that $W^u_{P_2}$ overshoots the equilibrium $P_1$, $P_1 < P_2$, if 
$W^u_{P_2} = (P(\xi), Q(\xi))$ intersects $L_{P_1}$ at finite $\xi_0$ and for all $\xi < \xi_0$ it holds $P(\xi) \in (P_1, P_2)$, $Q(\xi) < 0$. 
\item
We say that the invariant orbit $(P(\xi), Q(\xi))$ satisfying $Q = Q (P)$ lies under the invariant orbit  $(\widehat{P}(\xi), \widehat{Q}(\xi))$ satisfying $\widehat{Q} = \widehat{Q}(\widehat{P})$ if $Q(P) < \widehat{Q}(P)$ for all $P$ for which both $Q(P)$ and $\widehat{Q}(P)$ are defined. Note that \eq{e1} implies that this notation is well defined. 
\item
Finally, let $0 = P_1 < \dots < P_n = 1$ are all the equilibria of \eq{e1}--\eq{e2} in $[0,1]$. We say that $P_{\ell}$ is the last connected saddle to the saddle $P_m$ before $P_{k}$, $1 \le k < \ell < m \le n$, if $P_{\ell}$ and $P_m$ are saddle points, there exist a heteroclinic orbit from $P_m$ to $P_{\ell}$ for some $c$ but no heteroclinic orbit exists from $P_m$ to $P_r$ for any $c$ for all 
$r$, $k < r < \ell$. Note that $P_k$ in this definition can be a saddle or a node. 
\end{itemize}

\section{Existence of Traveling Waves}
\label{s:existence}
In this section we discuss existence of admissible heteroclinic orbits with $P \in [0,1]$ connecting $(0,0)$ and $(1,0)$ satisying \eq{TypeA} or \eq{TypeB} for a general class of nonlinearities $f(P)$ satisfying the conditions
\begin{itemize}[leftmargin=2\parindent]
\item[(S1)]
$f(P)$ is continuous for $P \in [0,1]$;
\item[(S2)]
$f(0) = f(1) = 0$;
\item[(S3)]
$f(P)$ has a finite number of zeros in $(0,1)$ and it has a non-zero derivative at each of its zeros in $[0,1]$;
\item[(S4)]
$f(P)$ is differentiable for $P \in (0,1)$.
\end{itemize}
The assumption (S4) is only technical and can be removed. The assumption (S3) on non-zero derivative at each zero in $[0,1]$ is often just technical, see \cite{Hou2010} for the treatment of the case of vanishing derivatives of $f(P)$ at its zeros, $f(P) > 0$ for $P \in (0,1)$ in the case of \eq{tFKPP}. Differentiability at zero points in (S3) can be alleviated even further but that requires a significant theoretical overhead as it would not be possible to use the standard results in the theory of dynamical systems, see \cite{Polacik2015} for a different approach that completely avoids this assumption.

\begin{figure}[t] 
\centering
\includegraphics[width=0.33\textwidth]{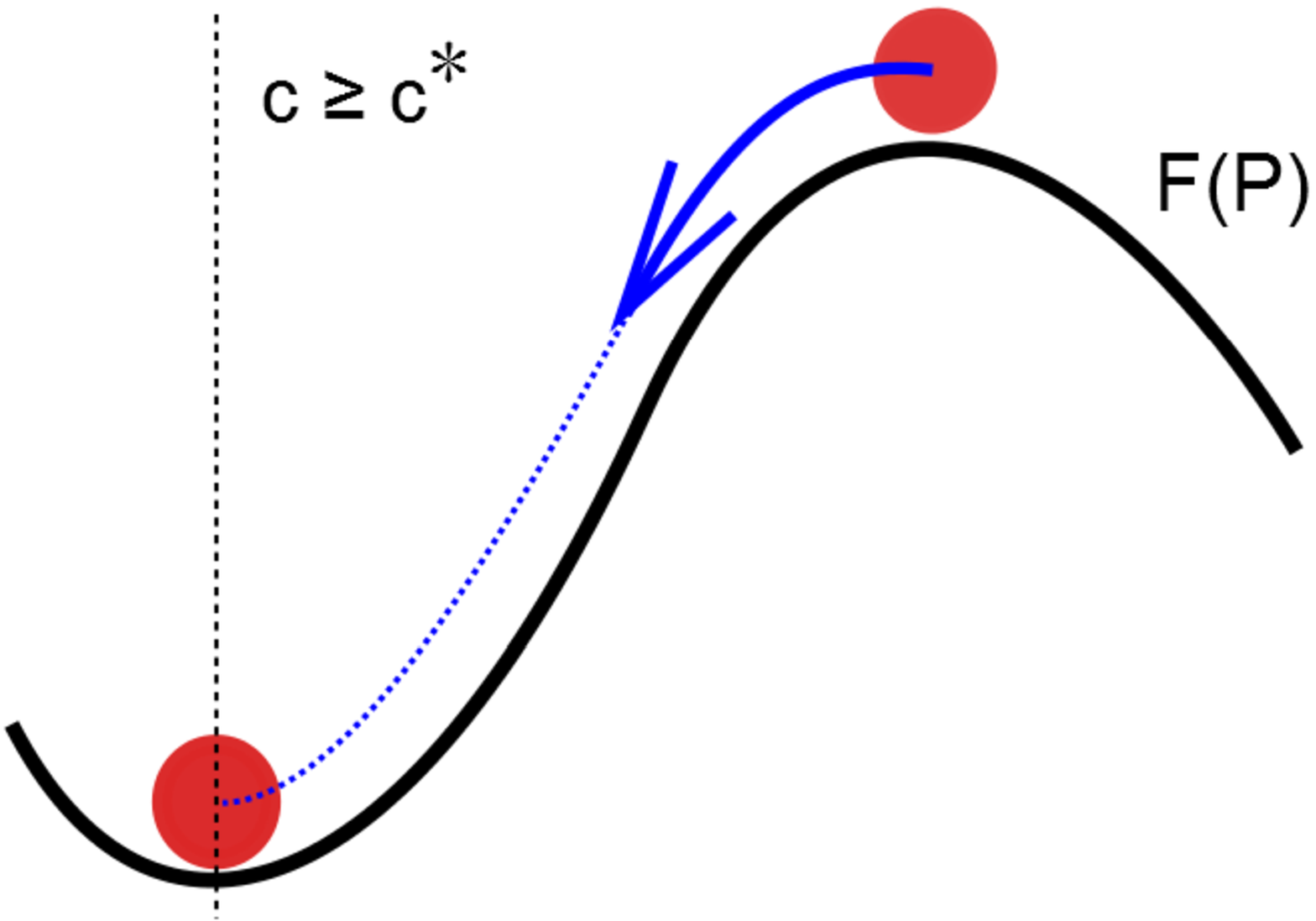}
\qquad\qquad \qquad
\includegraphics[width=0.33\textwidth]{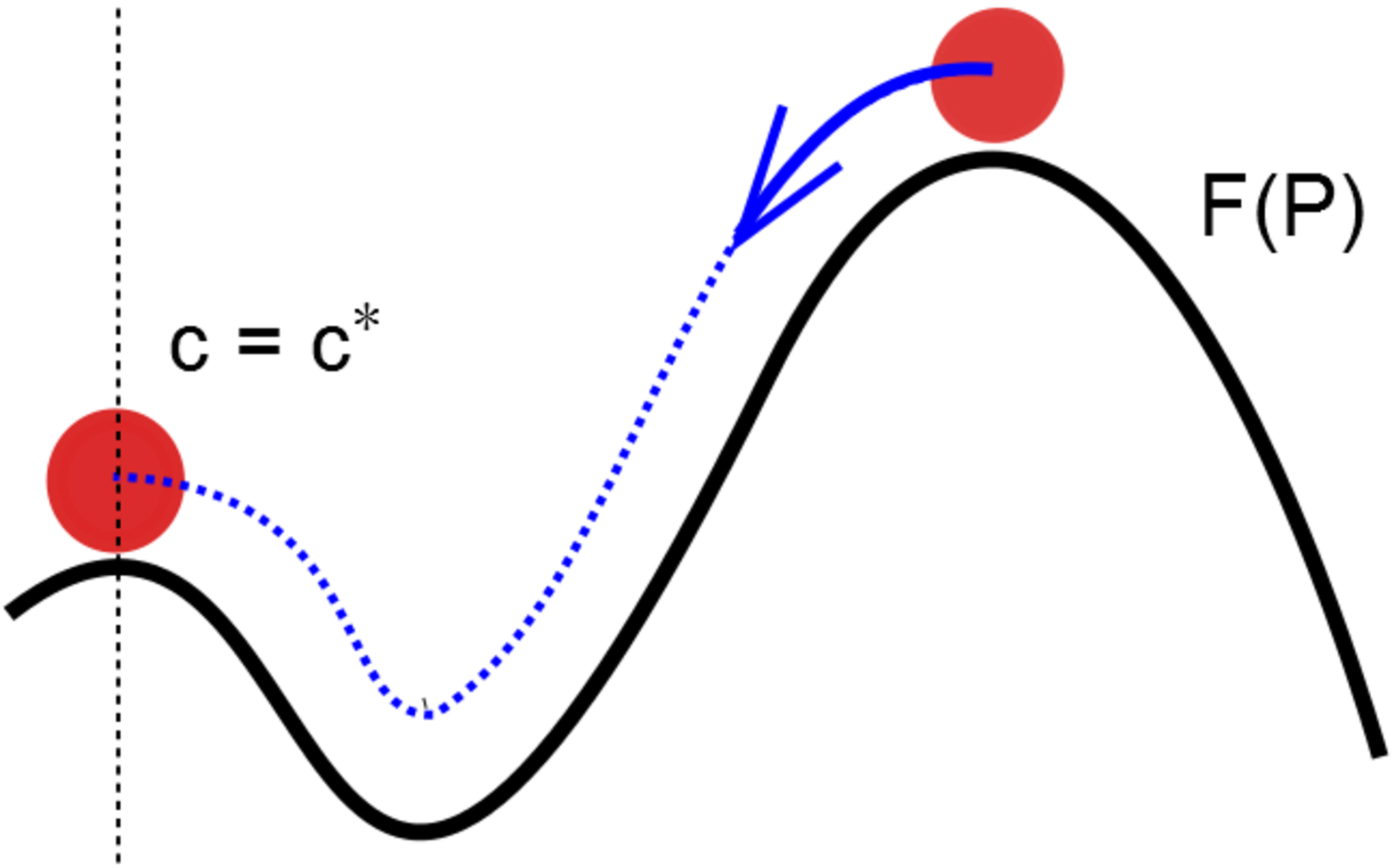}
\caption{Illustration of the mechanical analogue to \eq{tFKPP}. (Left panel) For a quadratic (monostable) $f(P)$ the trajectory from a saddle point to its neighboring node exists if and only if there is enough friction $c \ge c^{\ast}$. Less friction would yield decaying oscillations about the spiral node. (Right panel) For a cubic (bistable) $f(P)$ the trajectory from a saddle point to its neighboring saddle point exists if and only if the friction is equal to the critical friction $c = c^{\ast}$. More friction would yield decaying oscillations about the spiral node in between the saddle points while less friction would lead to overshooting of the saddle point. 
\label{Fig:min}}
\end{figure}

First, as a visual guidance for our results we will discuss the mechanical analogue intuition presented in \cite{Hadeler1987} (see also  \cite{Conley1978} and \cite{Smoller1982}) to determine the ranges of traveling wave speeds for which the traveling wave exists for the FKPP equation \eq{tFKPP}. The equation \eq{secondorder} for $M(P) \equiv 0$ and $D(P) \equiv 1$ can be interpreted as an equation for the position of a mass point on a surface (curve) of the potential energy $F(P) = \int f(P)\, dp$ with a (signed) friction with the magnitude $c$. Note that negative values of $c$ correspond to a physically unrealistic negative friction. 
To obtain such an interpretation multiply \eq{secondorder} by $P'(\xi)$ to obtain
$$
\frac{d}{d\xi} \left[ \frac{P'(\xi)^2}{2} +  F(P) \right] 
+ c (P'(\xi))^2 = 0\, .
$$
The traveling wave satisfying \eq{admsol2} can be then  interpreted as a trajectory from the saddle point of $F(P)$ 
at $P = 1$ to the saddle point or extremal point $P=0$, where $\xi$ becomes a time-like variable, although one has to keep in mind that it takes an infinite time to the mass point to get away from the hyperbolic saddle point at its initial position.

If $f(P) = k P(1-P)$ then the critical point at $P = 0$ is a minimum of $F(P)$ (see Fig.~\ref{Fig:min}, left panel). If the friction is too small $c \ll 0$, the trajectory of the mass point (starting at $P = 1$) will overshoot the equilibria at $P = 0$. For $c$ close to the critical speed $c^{\ast}$ the mass point will make decaying oscillations around $P = 0$ as $\xi \rightarrow \infty$. On the other hand, for supercritical  friction, $c \ge c^{\ast}$, the mass point will reach $P = 0$ as $\xi\rightarrow \infty$. Therefore, for the quadratic nonlinearity the traveling wave satisfying \eq{admsol2} exists if and only if $c \ge c^{\ast}$ for some critical wave speed $c^{\ast}$.

Analogously one can consider the case of a cubic nonlinearity $f(P) = k P (1-P) (P- P_0)$. In that case the trajectory should connect  the point of local maximum of the potential $F (P)$ at $P = 1$ with the neighboring local maximum at $P = 0$  (see Fig.~\ref{Fig:min}, right panel). Once again, if the friction is too small (the friction is very negative, i.e. the anti-friction is too large) $c \ll 0$ the trajectory will overshoot $P = 0$. On the other hand, for friction equal or larger than the critical value $c \ge c_{P_0}^{\ast}$ the trajectory will get trapped in the point of local minimum of $F(P)$ at $P = P_0$. If one lowers the value of friction below $c^{\ast}_{P_0}$ the trajectory will overshoot $P = P_0$. Smaller and smaller friction will extend the first oscillation of the trajectory further below $P = P_0$. By the continuity at some particular value $c = c^{\ast}$ the trajectory reaches $P = 0$ at infinite time. Since $P = 0$ is a point of local maximum of the potential, for any friction $c < c^{\ast}$ the trajectory of the mass point will overshoot $P = 0$ and the trajectory will go to $P = -\infty$ as $\xi \rightarrow \infty$. Therefore for the cubic nonlinearity $f(P) = k P (1-P) (P- P_0)$ the traveling wave of the type \eq{admsol2} exists only for  the wave speed $c = c^{\ast}$. 

\begin{figure}[t] 
\centering
\includegraphics[width=0.6\textwidth]{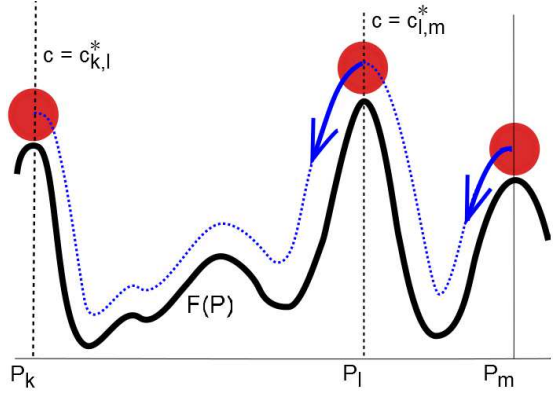}
\caption{Illustration of the mechanical analogue to \eq{tFKPP}. A saddle--saddle connection for $f(P)$ with more than two roots in between $P_m$ and $P_k$. While the saddle point $P_{\ell}$ is accessible from $P_m$ for $c = c^{\ast}_{\ell,m}$, the saddle points between $P_l$ and $P_k$ are not accessible from $P_{\ell}$. The connection from $P_{\ell}$ to $P_k$ exists for $c = c^{\ast}_{k,\ell}$. Then the connection from $P_m$ to $P_k$ then exists if and only if $c^{\ast}_{k,\ell} < c^{\ast}_{\ell,m}$ for 
$c =c^{\ast}_{k,m} \in (c^{\ast}_{k,\ell}, c^{\ast}_{\ell,m})$.
\label{Fig:stos}}
\end{figure}

Existence or nonexistence of the traveling waves for general nonlinearities $f(P)$ can be deduced by local analysis of trajectories from the point of local maximum of $F(P)$ to its neighboring local minimum and local maximum. We denote
$0 = P_1 < P_2 < \dots < P_n = 0$ the local extrema of $F(P)$ on $[0,1]$  (the zero points of $f(P)$),  where the local maxima (saddle points of the potential $F(p)$) are at $P_n$, $P_{n-2}, \dots$, and the local minima (nodes) at $P_{n-1}$, $P_{n-3}, \dots$ (see Fig.~\ref{Fig:stos}) 
The argument above guarantees that for any $k \ge 1$ the trajectory originating as $\xi \rightarrow -\infty$ at the saddle $P_{k+1}$ converging to the node $P_{k}$, $1 \le k < n$, as $\xi \rightarrow \infty$ exists if and only if $c \ge c^{\ast}_{k,k+1}$, and the trajectory originating as $\xi \rightarrow -\infty$ at the saddle $P_{k+1}$ converging to  the saddle $P_{k-1}$, $1 < k < n$, as $\xi \rightarrow \infty$ exists if and only if $c = c^{\ast}_{k-1,k+1}$.  Note that $c^{\ast}_{k-1,k+1} < c^{\ast}_{k,k+1}$ for all admissible $k$. 
Hence, it is easy to see (see \cite{Hadeler1987} for details) that there is a connection of the saddle point $P_m$ to the saddle point $P_{k}$, $1 \le k < m \le n$, if and only if the following recursive conditions are met:
\begin{itemize}[leftmargin=2\parindent]
\item[(R1)]
the connection of the saddle $P_m$ to the saddle $P_{\ell}$ exists for some $k < \ell  < m$ and  $c =  c^{\ast}_{\ell,m}$,
\item[(R2)]
a connection of the saddle $P_m$ to the saddle $P_s$ does not exist for any $s$, $k < s  < {\ell}$, for any $c$,
\item[(R3)]
the connection of the saddle $P_{\ell}$ to the saddle $P_{k}$ exists for $c = c^{\ast}_{k,\ell}$ and $c^{\ast}_{k,\ell} < c^{\ast}_{\ell,m}$.
\end{itemize}
Then the connection of $P_m$ to $P_k$ exists for $c = c^{\ast}_{k,m} \in (c^{\ast}_{k,\ell},c^{\ast}_{\ell,m})$.

Note that for $\ell = m - 2$ the first part of the condition is always met, i.e.,  for $k <  m-2$ there always exists $\ell$ satisfying (R1) and (R2). If $k = m-2$ the condition is empty and the traveling wave always exists for $c = c^{\ast}_{m-2,m}$. The situation is analogous when considering the connection of the saddle point $P_m$ to the stable node $P_k$ of $F(P)$.
One just needs to replace (R3) by  
\begin{itemize}[leftmargin=2\parindent]
\item[(R3')]  the connection  of the saddle $P_{\ell}$ to the stable node $P_{k}$ exists for 
all $c \in [c^{\ast}_{k, \ell},  c^{\ast}_{\ell,m})$. 
\end{itemize}
Then the connection of $P_m$ to $P_k$ exists for all $c \in [c^{\ast}_{k,m}, c^{\ast}_{\ell,m})$ where 
$c^{\ast}_{k,m} \in (c^{\ast}_{k, \ell},  c^{\ast}_{\ell,m})$.

\begin{figure}[t]
\centering
\includegraphics[width=\textwidth]{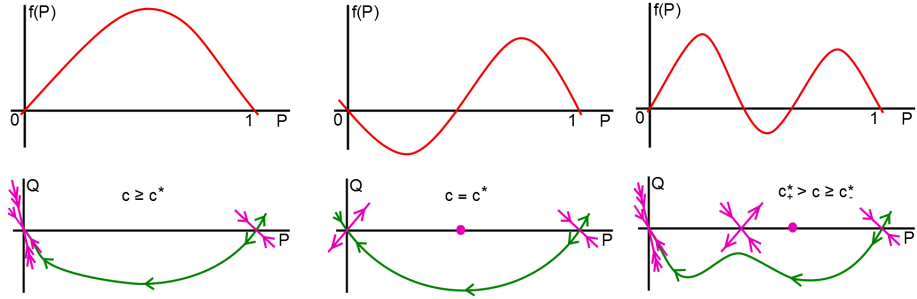}
\caption{Illustration of the results of Theorem~\ref{th:existence}. The upper plots show $F(P)$ vs.~$P$, the lower the existing connecting orbit and the range of $c$ for which it exists.}
        \label{fig:existence}
\end{figure}

We distinguish  six different cases of $f(P)$.  The results of Theorem~\ref{th:existence} are illustrated in Fig.~\ref{fig:existence}. 

\begin{theorem}
Let $f(P)$ satisfies the assumptions  (S1)--(S4), and let $D_1, D_2 > 0$.
Then an admissible heteroclinic orbit for the flow \eq{e1}--\eq{e2} exists if and only if $c$ satisfies:
\begin{itemize}[leftmargin=2\parindent]
\item[(A1)] 
If  $f(P) >0$ for $P \in (0,1)$ then $c \in [c^{\ast}, \infty)$  (orbits satisfying \eq{TypeA}), 
and $c \in (-\infty, c^{\ast\dagger}]$   (orbits satisfying \eq{TypeB}).
\item[(A2)] 
If $f(P) <0$ for $P \in (0,1)$ then  $c \in (-\infty,c^{\ast}]$ (orbits satisfying \eq{TypeA}), 
and $c \in [c^{\ast\dagger},\infty)$  (orbits satisfying \eq{TypeB}).
\item[(B)]
If $f'(0) f'(1)>0$ and if there is exactly one point $P_1 \in (0,1)$ such that $f(P_1) = 0$
then $c = c^{\ast}$ (orbits satisfying \eq{TypeA}), 
and  $c = c^{\ast\dagger}$ (orbits satisfying \eq{TypeB}). 
\item[(C1)]
If $f'(0) > 0$ and $f'(1)< 0$ and there are at least two distinct roots of $f(P)=0$ in $(0,1)$ then the set of values of $c$  is either a bounded interval $[c^{\ast}_-, c^{\ast}_+)$ or the empty set (orbits satisfying \eq{TypeA}), and it is a bounded interval  $(c^{\ast\dagger}_{-}, c^{\ast\dagger}_{+}]$ or the empty set (orbits satisfying \eq{TypeB}). 
\item[(C2)]
If $f'(0) < 0$ and $f'(1)> 0$ and there are at least two distinct roots of $f(P)=0$ in $(0,1)$ then the set of values of $c$  is either a bounded interval $(c^{\ast}_-, c^{\ast}_+]$ or the empty set (orbits satisfying \eq{TypeA}), and it is a bounded interval  $[c^{\ast\dagger}_{-}, c^{\ast\dagger}_{+})$ or the empty set (orbits satisfying \eq{TypeB}). 
\item[(D)]
If $f'(0)f'(1) > 0$ and there are at least three distinct roots of $f(P) = 0$ in $(0,1)$ either the orbit does not exists or it exists for $c = c^{\ast}$ for some $c^{\ast} \in \RR$ (orbits satisfying \eq{TypeA}), and it either does not exists or it exists for $c = c^{\ast\dagger}$ for some $c^{\ast\dagger} \in \RR$ (orbits satisfying \eq{TypeB}).

\end{itemize}
Furthermore, all these admissible heteroclinic orbits are monotone, i.e. 
$$
\mbox{either} \quad Q(z) = P'(z) \ge 0 \ \ \mbox{for all $z \in \RR$} \quad \mbox{or} \quad Q(z) = P'(z) \le 0\ \  \mbox{for all $z \in \RR$.}
$$ 
\label{th:existence}
\end{theorem}
 Note that in various cases it is possible to further specify lower or upper bounds for the critical speeds, see the proof for details. 
Theorem~\ref{th:existence} immediately follows from four lemmae that translate the intuition gained from the mechanical analogue described above.

\begin{lemma}
Assume that $f(P)$ satisfies assumptions (S1)--(S4) and that for $0 \le P_1 < P_2 \le 1$ 
$$
f(P_1) = 0 = f(P_2) \qquad
\mbox{and} \qquad
f(P) > 0 \quad\mbox{for all $P \in (P_1, P_2)$.}
$$
Then an heteroclinic orbit $(P(\xi), Q(\xi))$ for the flow \eq{e1}--\eq{e2} satisfying 
\begin{equation}
\lim_{\xi \rightarrow -\infty} (P(\xi),Q(\xi)) =  (P_2,0), \qquad \qquad \lim_{\xi \rightarrow \infty}(P(\xi),Q(\xi)) = (P_1,0),
\label{TypeC}
\end{equation}
 and  $P_1< P(\xi) < P_2$  for all $\xi \in \RR$ exists if and only if 
$c \ge c^{\ast}$ for some $c^{\ast} \in \RR$. 
\label{Lemma1}
\end{lemma}

\begin{proof}
First, using the linear analysis of \eq{e1}--\eq{e2} at its equilibrium $P_1$ we derive a necessary condition for existence of such an orbit. The condition $f'(P_1) > 0$ implies $\alpha_{P_1} > 0$. If $\beta_{P_1} >  0$ then $\Re \la_{P_1}^{\pm} > 0$ and the fixed point $P = P_1$ is a source and the orbit does not exit.
If $\beta_{P_1} \le 0$ the dynamics close to $P_1$ is determined by the sign of $\triangle = \frac{1}{4}(\beta_{P_1}^2 - 4 \alpha_{P_1})$. If $\triangle < 0$ then the equilibrium is a spiral sink and the 
heteroclinic orbit with $P(z) > P_1$  does not exist. If $\triangle \ge 0$, the two eigenvalues are real negative, $\la_{P_1}^- \ge \la_{P_1}^+ < 0$, and the equilibrium is a hyperbolic sink. Therefore the necessary conditions for existence of the heteroclinic orbit with the given properties are
$\beta_{P_1} \le 0$ and $\beta_{P_1}^2 \ge 4 \alpha_{P_1}$, that is equivalent to 
\begin{equation}
c \ge M(P_1) + 2\sqrt{f'(P_1)D(P_1)}\, .
\label{cest}
\end{equation}

At $P = P_2$ it holds $\alpha_{P_2} < 0$ and thus both roots of \eq{quadratic} are real and $\la_{P_2}^+ > 0 > \la_{P_2}^-$, i.e. the equilibrium $P_2$ is a saddle point. Therefore the heteroclinic orbit with the required properties must approach $P_2$ as $z \rightarrow -\infty$ and $P\rightarrow P_2^-$ along the unstable manifold $W^s_{P_2}$, i.e., in the direction of the eigenvector $(1,\la_1^+)$.

Next, we show that for $c \gg 1$ such an heteroclinic orbit exists. 
We construct a forward invariant region $\mathcal R$ in the phase space $(P,Q)$ with respect to the flow \eq{e1}--\eq{e2}. The region is bounded by line $Q =0$ from above and by the curve 
\begin{equation}
Q = h(P) = -f(P)/D(P)
\label{hp}
\end{equation}
 from below. Clearly $h(P)$ intersects $Q = 0$ at $P = P_1$ and $P = P_2$ and $h(P) < 0$ for all $P \in (P_1, P_2)$.  It is easy to see that for $Q = 0$ one has $P' = 0$ and $Q' < 0$, and thus the flow points inwards on the upper boundary of $\mathcal R$. On the other hand, at $(P, h(P))$ the flow points inwards (or tangentially) if and only if
$(P',Q')\cdot (h'(P), -1) \le 0$ since $(h'(P), -1)$ is the outer normal vector of the curve $Q = h(P)$. The condition can be written  after the division by $-h(P) > 0$ as
\begin{equation}
-h'(P) + \frac{M(P) - c}{D(P)} + 1 \le 0\, .
\end{equation}
Therefore if 
\begin{equation}
c \ge \max_{P\in [P_1, P_2]} \left\{ M(P) + D(P)\left( 1 + \frac{d}{dp} \, \frac{f(P)}{D(P)} \right) \right\}
\label{ccond}
\end{equation}
then the region $\mathcal R$ is forward invariant. Moreover, it is easy to check that for $c \gg 1$ the manifold $W^u_{P_2}$ lies locally (close to $P_2$) inside $\mathcal R$ as 
$$
\la^+_{P_2} \le - \frac{d}{dP}\, \left.\frac{f(P)}{D(P)}\right|_{P = P_2} \, .
$$
The last inequality can be rewritten as \footnote{check the end of inequality}
$$
M(P_2) - c + \sqrt{(M(P_2) - c)^2 - 4f'(P_2)D(P_2)} \le -2f'(P_2)\, ,
$$
and it is equivalent for $c > M(P_2)$ to \eq{ccond} with the right-hand side evaluated at $P = P_2$. 
Therefore $W^u_{P_2}$ lies inside $\mathcal R$, and it coincides with the heteroclinic orbit connecting $P_2$ to $P_1$ of the required properties. Thus for any $c$ satisfying both \eq{cest} and \eq{ccond} the traveling wave satisfying \eq{TypeC} exists. 
Note that the condition \eq{ccond} evaluated at $P = P_1$ implies
$$
c \ge M(P_1) + D(P_1) + f'(P_1) \ge M(P_1) + 2\sqrt{f'(P_1)D(P_1)}\, ,
$$
and therefore the condition \eq{ccond} implies \eq{cest}. Hence the traveling wave exists for all $c$ satisfying \eq{ccond}. Furthermore, if 
$f(P)$ is concave $f''(P) < 0$, $P \in [P_1, P_2]$, and if $D(P_2) \ge D(P_1)$  and $M_1 \le M_2$, the maximum in \eq{ccond} is attained at $P = P_1$ and then \eq{ccond} is equivalent to \eq{cest} and then \eq{cest} becomes both the sufficient and the necessary condition for the existence of the traveling wave.

Next, we show that if such a heteroclinic orbit exists for some $\hat{c}$ then it exists for all $c > \hat{c}$. Consider any such $c$. It is enough to notice that $d\la_{P_2}^+ / dc < 0$. Therefore the unstable manifold
$W^u_{P_2}(\hat{c})$ lies for $\xi \rightarrow -\infty$ below $W^u_{P_2}(c)$. But these two manifolds cannot intersect for any $P \in (P_1, P_2)$ as at any such eventual point $(P,Q)$ of intersection it is easy to see that $Q'(c) > Q'(\hat{c})$ and that makes the intersection impossible. Since $W^u_{P_2}(c)$ cannot intersect the line $Q = 0$, and $P' = Q < 0$, it must also converge to the fixed point $(P_1, 0)$ and it forms a heteroclinic orbit.
\end{proof}

\begin{lemma}
Assume that $f(P)$ satisfies assumptions (S1)--(S4) and that for $0 \le P_1 < P_2 < P_3 \le 1$ 
\begin{gather*}
f(P_1) = f(P_2) = f(P_3) = 0, \\
f(P) < 0 \quad\mbox{for all $P \in (P_1, P_2)$,}
\qquad
f(P) > 0 \quad\mbox{for all $P \in (P_2, P_3)$.}
\end{gather*}
Then a heteroclinic orbit $(P(\xi), Q(\xi))$ for the flow \eq{e1}--\eq{e2} satisfying
\begin{equation}
\lim_{\xi \rightarrow -\infty} (P(\xi),Q(\xi)) =  (P_3,0), \qquad \qquad \lim_{\xi \rightarrow \infty}(P(\xi),Q(\xi)) = (P_1,0),
\label{TypeD}
\end{equation}
such that  $P_1< P(\xi) < P_3$  for all $\xi \in \RR$ exists if and only if 
$c = c^{\ast}$. 
\label{Lemma2}
\end{lemma}

\begin{proof}
Since $f'(P_1) < 0$ and $f'(P_3) <0$ the eventual heteroclinic orbit is a saddle-saddle connection. 
First, we use the result of Lemma~\ref{Lemma1}. For $c \ge c^{\ast}_{23}$ the unstable manifold $W^u_{P_3}$ connects to the equilibrium $(P_2,0)$ that is a hyperbolic sink. Analogously, using the symmetries \eq{sym1} and \eq{sym2}  for $c < -c^{\ast}_{12}$ the stable manifold $W^s_{P_1}$ connects to $(P_2,0)$ that is a hyperbolic source. 

Now consider  $c \in (-c^{\ast}_{12}, c^{\ast}_{23})$. 
At $Q = 0$  one has $Q' > 0$  for $P \in (P_1,P_2)$ and $Q' < 0$ for $P\in (P_2, P_3)$. Since $P' < 0$  for $Q < 0$, both manifolds $W^u_{P_3}$ and $W^s_{P_1}$ must intersect the half-line $L_{P_2}$. Let us denote by $(P_2,Q_3(c))$ and $(P_2,Q_1(c))$,  the first (with respect to $\xi$) intercept of $W^u_{P_3}$  and the last intercept of $W^s_{P_1}$ with $L_{P_2}$, respectively. These intercepts continuously depend on the parameter $c$, i.e., the function $w(c) = Q_3(c) - Q_1(c)$ is continuous. Also 
by continuity $Q_3(c) \rightarrow 0^-$ and $Q_1(c) \nrightarrow 0$  as $c \rightarrow (c^{\ast}_{23})^-$ 
and $Q_1(c) \rightarrow 0^-$ and $Q_3(c) \nrightarrow 0$ as $c \rightarrow -(c^{\ast}_{12})^+$. Hence $w(c^{\ast}_{23})w(-c^{\ast}_{12}) < 0$ and the function $w(c)$ must have a root inside the interval $(-c^{\ast}_{12}, c^{\ast}_{23})$. 
But that means that the invariant manifolds $W^u_{P_3}$ and $W^s_{P_1}$ connect at $P = P_2$ and therefore for such a $c$ there exists a heteroclinic orbit satisfying \eq{TypeD}. The property $P_1< P(\xi) < P_3$ follows from the fact the $P' = Q < 0$ for $Q < 0$. 

It remains to prove that such a $c$ is unique. We prove this claim by a contradiction. 
Assume that there are two heteroclinic orbits $(P_{c_1}(\xi),Q_{c_1}(\xi))$ and $(P_{c_2}(\xi), Q_{c_2}(\xi))$ for $c_1 < c_2$. 
These orbits coincide with the stable manifolds of the equilibrium $(P_1,0)$ that asymptotically approach as $\xi \rightarrow \infty$ 
the stable manifolds of the linearized flow near $P_1$ given by the lines $Q = \la^-_{c_1}(P-P_1)$ and $Q = \la^-_{c_2}(P-P_1)$, $P > P_1$. 
Analogously, these orbits coincide with the unstable manifolds of the equilibrium $(P_3,0)$ that asymptotically 
approach the unstable manifolds of the linearized flow near $(P_3,0)$ given by the lines $Q = \la^+_{c_1}(P-P_3)$ and $Q = \la^+_{c_2}(P-P_3)$, $P < P_3$. 
These eigenvalues are given for $j=1,2$ by the formulae
\begin{eqnarray}
\la^-_{c_j} & = & \frac{1}{2D(P_1)} \left( M(P_1) - c_j - \sqrt{(M(P_1) - c_j)^2 - 4F'(P_1)D(P_1)}\right), 
\nonumber \\
\la^+_{c_j} & = & \frac{1}{2D(P_3)} \left( M(P_3) - c_j + \sqrt{(M(P_3) - c_j)^2 - 4F'(P_3)D(P_3)}\right)\, . 
\nonumber
\end{eqnarray}
Since both functions $y = x \pm \sqrt{x^2 + b}$ are increasing for $b >0$ and also $M_i- c_1 > M_i - c_2$ 
for $i=1,2$, the eigenvalues satisfy
$$
\la^-_{c_1} < \la^-_{c_2} < 0 < \la^+_{c_1} < \la^-_{c_2}\, . 
$$
Therefore, the orbits $(P_{c_1},Q_{c_1})$ and $(P_{c_2}, Q_{c_2})$ must have a nondegenerate intersection at which they satisfy  $Q_{c_2}' > Q_{c_1}'$. 
But at any common point $(\widehat{P}, \widehat{Q})$ of the phase plane the flows are given by 
$$
P_{c_j}' = \widehat{Q}, \qquad
Q_{c_j}' = \frac{M(\widehat{P})}{D(\widehat{P})} \, \widehat{Q} - \frac{f(\widehat{P})}{D(\widehat{P})} - \frac{\widehat{Q}}{D(\widehat{P})} \, c_j \, ,
$$
and hence 
$Q_{c_1}' > Q_{c_2}'$ yielding a contradiction. 
Note that we have also proved that $c^{\ast} < c^{\ast}_{23}$.
\end{proof}

\begin{lemma}
Assume that $f(P)$ satisfies assumptions (S1)--(S4). 
Let $0 = P_1 < \dots < P_{n} = 1$, $n \ge 5$, be the zero points of $f(P)$ in $[0,1]$. 
Furthermore let $k, \ell, m$ be nonnegative integers such that 
$1 \le k < \ell < m \le n$ and 
$$
f'(P_k) < 0, \qquad f'(P_{\ell})< 0, \qquad f'(P_m) < 0,
$$
and that for the flow \eq{e1}--\eq{e2}
\begin{itemize}
\item
 the heteroclinic orbit from $P_m$ to $P_{\ell}$ exists for $c = c^{\ast}_{\ell, m}$;
\item
$P_{\ell}$ is the last connected saddle to $P_{m}$ before $P_{k}$;
\item
the heteroclinic orbit from $P_{\ell}$ to $P_{k}$ exists for $c = c^{\ast}_{k, \ell}$.
\end{itemize}
Then a saddle-saddle heteroclinic orbit  $(P(\xi),Q(\xi))$ satisfying
\begin{gather}
\lim_{\xi \rightarrow -\infty} (P(\xi),Q(\xi)) =  (P_{m}^-,0^-), \qquad
 \lim_{\xi \rightarrow \infty}(P(\xi),Q(\xi)) = (P_{k}^+,0^-), \label{TypeDa}\\
\mbox{$P_{k} < P(\xi) < P_{m}$  for all $\xi \in \RR$} 
\label{TypeDb}
\end{gather}
exists if and only if 
$$
c^{\ast}_{k, \ell} <  c^{\ast}_{\ell, m}, \qquad \mbox{and} \qquad 
c = c^{\ast}_{k,m} \in (c^{\ast}_{k, \ell},  c^{\ast}_{\ell, m}).
$$ 
\label{Lemma3}
\end{lemma}

\begin{proof}
First, if $c^{\ast}_{k, \ell}  \ge  c^{\ast}_{\ell, m}$ then for 
$c \ge  c^{\ast}_{\ell, m}$  the unstable orbit $W^u_{P_{m}}$  converges to $P_{\ell}$ or does not reach $P_{\ell}$ and thus a heteroclinic orbit from $P_{m}$ to $P_{k}$ satisfying the conditions of the Lemma does not exists. On the other hand, for 
$c <  c^{\ast}_{\ell, m}$ the orbit $W^u_{P_{m}}$ intersects $L_{P_{s}}$ and thus it lies at $P = P_{\ell}$ under the heteroclinic orbit connecting $P_{\ell}$ to $P_{k}$. Since the heteroclinic orbit is invariant with respect to the flow, and $P'= Q <0$ in the lower half-plane of the phase plane $(P,Q)$, the orbit $W^u_{P_m}$ will remain under the heteroclinic orbit on the whole interval $P \in [P_k, P_{\ell}]$ and thus the heteroclinic connection from  $P_m$ to $P_k$ does not exist for any $c$. 

However, situation is different if  $c^{\ast}_{k, \ell}  <  c^{\ast}_{\ell, m}$. Then one can compare the location of intersects of
$W^u_{P_m}$ and $W^s_{P_k}$ with $L_{P_s}$ for all $c \in (c^{\ast}_{k, \ell},  c^{\ast}_{\ell, m})$. 
An argument analogous to proof of Lemma~\ref{Lemma2} then concludes the proof of the Lemma. 
\end{proof}

\begin{lemma}
Assume that $f(P)$ satisfies assumptions (S1)--(S4). 
Let $0 = P_1 < \dots < P_{n} = 1$, $n \ge 4$, be the zero points of $f(P)$ in $[0,1]$. 
Furthermore let $k, \ell, m$ be nonnegative integers such that 
$1 \le k < \ell < m \le n$ and 
$$
f'(P_k) > 0, \qquad f'(P_{\ell})< 0, \qquad f'(P_m) < 0,
$$
and that for the flow \eq{e1}--\eq{e2}
\begin{itemize}
\item
 the heteroclinic orbit from $P_m$ to $P_{\ell}$ exists for $c = c^{\ast}_{\ell, m}$;
\item
$P_{\ell}$ is the last connected saddle to $P_m$ before $P_k$;
\item
the heteroclinic orbit from $P_{\ell}$ to $P_{k}$ exists for $c \ge c^{\ast}_{k, \ell}$.
\end{itemize}
Then a saddle-node heteroclinic orbit  $(P(\xi),Q(\xi))$ satisfying
\eq{TypeDa}--\eq{TypeDb} exists if and only if 
$$
c^{\ast}_{k, \ell} <  c^{\ast}_{\ell, m}, \qquad \mbox{and} \qquad 
c \in [c^{\ast}_{k,m}, c^{\ast}_{\ell, m}), \qquad \mbox{where} \quad
c^{\ast}_{k,m} \in (c^{\ast}_{k, \ell},  c^{\ast}_{\ell, m}).
$$ 
\label{Lemma4}
\end{lemma}

\begin{proof}
If $c^{\ast}_{k, \ell}  \ge  c^{\ast}_{\ell, m}$ then for 
any $c \ge  c^{\ast}_{\ell, m}$  the unstable orbit $W^u_{P_{m}}$  converges to $P_{\ell}$ or it does not reach $P_{\ell}$ and thus a heteroclinic orbit from $P_{m}$ to $P_{k}$ satisfying the conditions of Lemma does not exists. On the other hand, for 
any $c <  c^{\ast}_{\ell, m}$ the orbit $W^u_{P_{m}}$ intersects $L_{P_{\ell}}$ and thus it lies at $P = P_{\ell}$ under the heteroclinic orbit connecting $P_{\ell}$ to $P_{k}$. Since the heteroclinic orbit is invariant with respect to the flow, and $P'= Q <0$ in the lower half-plane of the phase plane $(P,Q)$, the orbit $W^u_{P_m}$ will remain under the heteroclinic orbit on the whole interval $P \in [P_k, P_{\ell}]$ and thus the heteroclinic connection from  $P_m$ to $P_k$ does not exist for any $c$. 

However, situation is different if  $c^{\ast}_{k, \ell}  <  c^{\ast}_{\ell, m}$. Then one can compare the location of intersects of $W^u_{P_m}$ and $W^{ss}_{P_k}$ with $L_{P_{\ell}}$ for all $c \in (c^{\ast}_{k, \ell},  c^{\ast}_{\ell, m})$. 
For $c = c^{\ast}_{\ell, m}$ the manifold $W^u_{P_{m}}$ coincides with the heteroclinic orbit connecting $P_m$ to $P_{\ell}$,
the manifold $W^u_{P_{\ell}}$ coincides with the heteroclinic orbit connecting $P_{\ell}$ to $P_{k}$ and 
the manifold $W^{ss}_{P_k}$ intersects $L_{P_{\ell}}$ at some $Q< 0$ (as it lies under $W^u_{P_{\ell}}$. 
Therefore, by continuity $W^{ss}_{P_k}$ lies under $W^u_{P_{m}}$ for $c \in (c^{\ast}_{\ell, m}-\eps, c^{\ast}_{\ell, m})$ for some small $\eps >0$. But at the same time $W^u_{P_{\ell}}$ for all such $c$ connects to $P_{k}$. Therefore, $W^u_{P_m}$ also connects to $P_k$ and thus the heteroclinic orbit from $P_m$ to $P_k$ exists. 
On the other hand, for $c = c^{\ast}_{k, \ell}$ the manifold $W^u_{P_m}$ lies under the heteroclinic orbit connecting $P_{\ell}$ to $P_k$ that coincides with $W^u_{P_{\ell}}$ and $W^{ss}_{P_k}$ and thus the heteroclinic orbit from $P_m$ to $P_k$ does not exists.

Now assume that two admissible heteroclinic orbits of type \eq{admsol1} exist for $c_1$ and $c_2$, $c_1 < c_2$. The comparison
argument shows that for any $c \in (c_1, c_2)$ the unstable manifold $W_{P_m}^u(c)$ lies below the unstable manifold $W^u_{P_m}(c_2)$
and above $W^u_{P_m}(c_1)$. Since both these manifolds connect to $(0,0)$ also $W^u_{P_m}(c)$ is a heteroclinic orbit. Therefore the set of
$c$ for which there exist a heteroclinic connection satisfying \eq{TypeA} is a connected set, i.e. an interval.  On the other hand, from the proof it follows that the interval, if non-empty, has the form $[c^{\ast}, c^{\ast\ast})$. 

\end{proof}

Now we proceed with the proof of   Theorem~\ref{th:existence}. 
\begin{proof}
We give proof of existence of the connecting orbits of type \eq{TypeA}. 
Existence of orbits of type \eq{TypeB} in all cases follows by an application of the symmetry \eq{sym2}. 
For the sake of clarity of the argument we first present the proof of part (D) before (C1) and (C2). 

{\bf (A1)}
The existence of connecting orbits of type \eq{TypeA} follows from Lemma~\ref{Lemma1} for $P_1 = 0$ and $P_2 = 1$. 

{\bf (B)}
The existence of connecting orbits of type \eq{TypeA} follows  from Lemma~\ref{Lemma2} for $P_1 = 0$ and $P_3 = 1$. 

{\bf (D)}
We denote by $0 = P_1 < P_2 < \dots < P_{m} = 1$ the roots of $f(P)$ in the interval $[0,1]$. 
Let $P_s$ be the last connected saddle to $P_m$ before $P_1$. Existence of such $s$ follows by Lemma~\ref{Lemma2}. 
We set 
$s_0 = m$, $s_1 = s$, $c^{\ast}_0 = c^{\ast}_{s,m}$ 
and define a recurrent decreasing sequence $P_{s_j}$, $j \ge 1$, of equilibria in the following way: 
if a heteroclinic connection from the saddle point $P_{s_j}$ to $P_1$ does not exist then $P_{s_{j+1}}$ is the last connected saddle from $P_{s_j}$ before $P_1$; we also denote $c^{\ast}_j$ the wave speed $c$ for which the connection from $P_{s_{j}}$ 
to $P_{s_{j+1}}$ exists. If the orbit from $P_{s_j}$ to $P_1$ exist we set $s_{j+1} = 1$ 
and $c^{\ast}_{j+1}$ is the wave speed for which the heteroclinic orbit exists. By Lemma~\ref{Lemma2} this sequence is well defined 
and finite. Note that by Lemma~\ref{Lemma3} we have either $s_2 = 1$ or
$c^{\ast}_0 < c^{\ast}_1 < \dots < c^{\ast}_{j+1}$ as otherwise there would be a contradiction with the recurrent definition of~$P_{s_{j+1}}$. 

We distinguish two cases depending on whether the heteroclinic orbit connecting the saddle point 
$P_{s}$ to the saddle point $P_1 = 0$ exists or not, i.e., whether $s_2 = 1$ or $s_2 > 1$. If it exists then the claim of the Theorem follows immediately from Lemma~\ref{Lemma3} and the condition characterizing the existence of the connecting orbit is $c^{\ast}_2 = c^{\ast}_{1,s} < c^{\ast}_{s,m} = c^{\ast}_1$ and the orbit exists if and only if 
$c^{\ast} = c^{\ast}_{1,m} \in (c^{\ast}_{1,s}, c^{\ast}_{s,m})$. 

On the other hand we show that if $s_2 > 1$ the heteroclinic orbit does not exists.  
For $c \ge c^{\ast}_1$ the manifold $W^u_{m}$ does not reach $P_{s}$ and thus the heteroclinic orbit from $P_{m}$ to $P_1$ does not exists. For  $c < c^{\ast}_1$ the manifold  $W^u_{m}$  lies under the heteroclinic orbits connecting 
$P_{s_r}$ to $P_{s_{r+1}}$ for all $r = 0,\dots, j$, thus it will intersect $L_0$ at $Q< 0$. Hence a heteroclinic orbit from $P_{m}$ to $P_1$ does not exists. 

{\bf (C1)}
The argument is analogous to (D). We denote by $0 = P_1 < P_2 < \dots < P_{m} = 1$ the roots of $f(P)$ in the interval $[0,1]$.
Let $P_s$ be the last connected saddle to $P_{m}$ before $P_1$. Existence of such $s$ follows by Lemma~\ref{Lemma2}. 
We set $s_0 = m$, $s_1 = s$, $c^{\ast}_0 = c^{\ast}_{s,m}$ 
and define a recurrent decreasing sequence $P_{s_j}$, $j \ge 1$, of equilibria in the following way: 
if a heteroclinic connection from $P_{s_j}$ to $P_1$ does not exist for any $c$ then $P_{s_{j+1}}$ is the last connected saddle from
$P_{s_j}$ before $P_1$; we also denote $c^{\ast}_j$ the wave speed $c$ for which the connection from $P_{s_{j}}$ 
to $P_{s_{j+1}}$ exists. If the orbit from $P_{s_j}$ to $P_1$ exist we set $s_{j+1} = 1$ and
$c^{\ast}_{j+1}$ is the minimum wave speed for which the heteroclinic orbit from $P_{s_j}$ to $P_1$ exists. By Lemmae~\ref{Lemma1} and \ref{Lemma2} this sequence is well defined and finite.
Note that by Lemma~\ref{Lemma4} we have either $s_2 = 1$ or  $c^{\ast}_0 < c^{\ast}_1 < \dots < c^{\ast}_{j+1}$ 
as otherwise there would be a contradiction with the recurrent definition of $P_{s_{j+1}}$. 

We distinguish two cases depending on whether the heteroclinic orbit connecting the saddle point 
$P_{s}$ to $P_1 = 0$ exists or not, i.e., whether $s_2 = 1$ or $s_2 > 1$. If it exists then the claim of the Theorem follows immediately from Lemma~\ref{Lemma4} and the condition characterizing the existence of the connecting orbit is $c^{\ast}_2 = c^{\ast}_{1,s} < c^{\ast}_{s,m} = c^{\ast}_1$ and the orbit exists if and only if 
$c \in [c^{\ast}_{1,m}, c^{\ast}_{s,m})$ where $c^{\ast}_{1,m} \in (c^{\ast}_{1,s}, c^{\ast}_{s,m})$. 
On the other hand the same argument as in (D) shows that if $s_2 > 1$ the heteroclinic orbit from $P_m$ to $P_1$ does not exists.  

{\bf (A2), (C2)}
The results follow from (A1) and (C1) by  an application of the symmetry \eq{sym1}.
\end{proof}

\section{Quadratic Nonlinearity with Homogeneous Diffusion}
\label{s:homogeneous}
While in general we are not able to provide a formula for the range of speeds $c$ for which the traveling wave for the system \eq{FKPP} exists, it the special case of a quadratic (monostable) nonlinearity 
\begin{equation}
f(P) = k P (1-P)\, , \qquad k > 0\, ,
\label{Fquadratic}
\end{equation}
and equal diffusion coefficients it is possible to specify it completely.

\begin{theorem}
Let $f(P)$ be given by \eq{Fquadratic} and $D = D_1 = D_2$. Then there exists a traveling wave profile satisfying \eq{secondorder} and
\eq{TypeA} {\bf if and only if} $c \ge c^{\ast} (M_1, M_2, D, k)$ where
\begin{equation}
c^{\ast} = 
\begin{cases}
M_1 + 2\sqrt{kD}\, , \qquad 
& \mbox{if $M_2 \le M_1 + 2\sqrt{kD}$,} \\
\displaystyle\frac{M_1+ M_2}{2} + \displaystyle\frac{2kD}{M_2- M_1}\, , \qquad 
& \mbox{if $M_2 \ge  M_1 + 2\sqrt{kD}$.}
\end{cases}
\end{equation}
\label{th:homo}
\end{theorem}

The resulting formula for the critical minimal speed $c^{\ast}$ of the wave has a simple interpretation. Without loss of generality we can set $M_1 = 0$ as otherwise we may just consider the system in the reference frame traveling with speed $M_1$. If $M_2 \le c_{pull} = 2\sqrt{kD}$ then the drift of the genotype with the advection $M_2$ is slower then the pulling speed $c_{pull}$ of the other genotype and   the traveling wave will be pulled with the speed $c_{pull}$. The value of $M_2$ will influence only the shape of the wave with the transition zone between the bulk of the wave ($P\approx 1$) and its tail ($P \approx 0$) being narrower as $M_2 - M_1$ approaches $c_{pull}$.  
On the other hand, once the advection speed of the bulk $M_2$ becomes supercritical, i.e., larger than the pulling speed, the pushing speed of the bulk of the wave will overtake  through the action of diffusion the pulling speed in the tail region  and the wave will be traveling faster, i.e., there will be no admissible traveling waves with speed $c_{pull}$. The diffusion that mediates the influence of the bulk of the wave to its tail will attenuate the advection of the bulk and thus the minimal wave velocity $c^{\ast} \in (c_{pull}, M_2)$. 
The magnitude of the speed up of the wave due to the drift difference $M_2 - M_1$ over the averaged expected mean drift $(M_1 + M_2)/2$ is inversely proportional to $M_2 - M_1$ and it is equal to $2kD/(M_2 - M_1)$ that converges to $0^+$ as $M_2 - M_1 \rightarrow \infty$. 

The strong nonlinear dependence of $c^{\ast}$ on $M_2$ is in a strong contrast with the speed of the wave for the cubic 
$$
f(P) = kP(1-P)(P-\widehat{P}), \qquad 
\widehat{P} \in (0,1).
$$
In that case (see \cite{NK2016}) the unique traveling wave satisfying \eq{TypeA} is given by ($\xi = x - ct$)
\begin{equation*}
P(\xi) = \left[1 + \exp \left(\frac{k\zeta\xi}{2}\right)\right]^{-1}\, , 
\quad
\zeta =\frac{4}{-(M_2 - M_1) + \sqrt{(M_2 - M_1)^2 + 8 kD}}\, ,
\end{equation*}
and the wave travels with the speed
\begin{equation}
c_{cubic}^{\ast} =  \frac{M_1+M_2}{2} +
\frac{2kD (1-2\widehat{p})}{M_2 - M_1 \pm \sqrt{(M_2 - M_1)^2 + 8 kD}}\, .
\label{cubicspeed}
\end{equation}
Particularly note that if $\widehat{p} = 0.5$ then the speed of the wave depends linearly on $M_2$. 

\begin{proof}
The proof is based on two observations. 
The first observation is that for 
\begin{equation} 
c = \widehat{c} = 
\frac{M_1+ M_2}{2} + \frac{2f'(0)D}{M_2- M_1}  = 
\frac{M_1+M_2}{2} - \frac{2f'(0)D}{M_1- M_2}\, ,
\label{cast}
\end{equation}
it is easy to see that 
$$
(M_1 - \widehat{c}\, )^2 - 4f'(0)D = \left( \frac{M_1 - M_2}{2} - \frac{2f'(0)D}{M_1 - M_2}\right)^2
$$
and thus the two eigenvalues of the linearization of \eq{e1}--\eq{e2} at the origin $\la^-_0 \le \la^+_0 < 0$ are given by 
\begin{equation}
\la_0^{\pm} \in \left\{ \frac{M_1- M_2}{2D}, \, \frac{2f'(0)}{M_1 - M_2}\right\}\, .
\label{laeq}
\end{equation}
Let $\widehat{\la} = (M_1-M_2)/2D$ be one of the eigenvalues in \eq{laeq}.
We will show that for $f(P)$ satisfying  \eq{Fquadratic} the function uniquely determined in the phase plane $(P,Q)$ by
\begin{equation}
P' = Q =  \widehat{\la}\, \frac{f(P)}{f'(0)}\, 
\label{sol}
\end{equation}
is a solution%
\footnote{Compare with the function $h(p)$ given by \eq{hp} used in the proof of Lemma~\ref{Lemma1}.}
 of \eq{e1}--\eq{e2} for $c = \widehat{c}$.
Therefore by Lemma~\ref{Lemma1} the solution of \eq{secondorder} satisfying  \eq{TypeA} exists for all 
$c \ge \widehat{c}$.  Furthermore, we show that $\la_0^- = \widehat{\la}$  for $M_2 \ge M_1 + 2\sqrt{f'(0)D}$. Therefore the solution given by \eq{sol} coincides with the fast stable manifold $W^{ss}_0(\widehat{c})$ of $(0,0)$.

Next we show that if $\la_0^- = \widehat{\la}$ and $c \in [M_1 + 2\sqrt{f'(0)D}, \widehat{c}\, )$ then there is no solution of \eq{e1}--\eq{e2} satisfying \eq{TypeA}.  Denote by $(P_{\widehat{c}}, Q_{\widehat{c}})$ the heteroclinic connection for $c = \widehat{c}$ given by \eq{sol}, and consider the corresponding  unstable manifold $W^u_1(c)$ of the equilibrium $(1,0)$.
It is easy to see that 
$$
\frac{d}{dc}\la^-_0(c) <0, \qquad
\frac{d}{dc}\la^+_1(c) < 0,
\qquad
\mbox{for all $c \in [M_1 + 2\sqrt{f'(0)D}, \widehat{c} \, )$.}
$$ 
Indeed
\begin{eqnarray}
\la^-_0(c) & =&  \frac{1}{D}\left( M_1 - c - \sqrt{(M_1-c)^2 - 4f'(0)D}\right)\, , \label{ladefa}\\
\la^+_1(c)& =& \frac{1}{D}\left( M_2 - c + \sqrt{(M_2-c)^2 + 4f'(0)D}\right)\, ,
\label{ladefb}
\end{eqnarray}
where the function $g^-(x) = x - \sqrt{x^2 - h}$ is increasing for $x < -\sqrt{h}$  and the function $g^+(x) = x + \sqrt{x^2 + h}$ is increasing for all $x$. Therefore
\begin{equation}
\la_0^-(c) >  \la_0^-(\widehat{c})\, ,\qquad 
\la^+_1(c) > \la_1^+ (\widehat{c})\, .
\label{compla}
\end{equation}
The fast stable manifold $W^{ss}_0(c)$ of $(0,0)$ forms a separatrix in the region $\{(P,Q); 0 \le P \le 1, Q < 0\}$ of the phase space $(P,Q)$ of trajectories lying below $W^{ss}_0(c)$ and intersecting $L_0$ at some $Q < 0$ and trajectories lying above $W^{ss}_0(c)$ converging to $(0,0)$ as $\xi \rightarrow \infty$ and intersecting $Q = 0$ at some $P > 0$. 
Therefore if $W^u_1(c)$ coincides with the admissible heteroclinic orbit connecting $(1,0)$ to $(0,0)$, then it must lie on or above $W^{ss}_0(c)$. Close to $(0,0)$ 
the manifold $W^u_1(c)$ lies by \eq{compla} above $(P_{\widehat{c}}, Q_{\widehat{c}})$ and close to $(1,0)$ it lies below $(P_{\widehat{c}}, Q_{\widehat{c}})$. 
But since $c < \widehat{c}$ the slopes of these two curves in the phase space $(P,Q)$ satisfy at any point of their intersection the inequality
\begin{equation}
\frac{dQ_c}{dP} = \frac{1}{D} \left[ (M(P) - c) - \frac{f(P)}{Q}\right] >
 \frac{1}{D} \left[ (M(P) - \widehat{c}\, ) - \frac{f(P)}{Q}\right] = \frac{dQ_{\widehat{c}}}{dP}
\, ,
\nonumber
\end{equation}
yielding a contradiction with the assumption that $W^u_1(c)$ coincides with the admissible heteroclinic orbit connecting $(1,0)$ to $(0,0)$.

The second important observation is that the region
$$
\mathcal R = \{ (P,Q); P \in [0,1], g(P) \le Q \le 0,g(0) =  g(1) = 0, g(P) < 0, P \in (0,1)\}\, ,
$$
where 
$$
g(P) = \la^-_0 \, \frac{f(P)}{f'(0)}\, ,
$$ 
is forward invariant%
\footnote{Compare with \eq{sol} and \eq{hp}.}
with respect to the flow \eq{e1}--\eq{e2} for all $c \ge c^{\ast}$ and simultaneously, the unstable manifold $W^u_1$ lies locally around $(1,0)$ inside $\mathcal R$. Consequently, $W^u_1$ connects to $(0,0)$ as $\xi \rightarrow \infty$ and it forms an admissible heteroclinic orbit satisfying \eq{TypeA}. 

On the upper boundary of $\mathcal R$ corresponding to $Q = 0$ and $P \in (0,1)$ one has $P' = Q = 0$ and $Q' = -f(P)/D < 0$, i.e. the vector field points inwards.  Similarly as in the proof of Lemma~\ref{Lemma1} we derive the condition for the flow pointing inward on the lower boundary of $\mathcal R$ parametrized by $(P,g(P))$:
\begin{equation}
g(P) g'(P) - \frac{M(P)-c}{D} g(P) + \frac{f(P)} {D} \le 0\, .
\label{outcond}
\end{equation}
It is equivalent to 
$$
(\la^-_0)^2 \frac{f(P)f'(P)}{f'(0)^2} - \frac{M_1-c}{D}\,\la^-_0\, \frac{f(P)}{f'(0)} 
- \frac{M_2-M_1}{D} \,\la^-_0 \, P \,\frac{f(P)}{f'(0)}  + \frac{2f(P)}{D} \le 0\, ,
$$
and furthermore to 
$$
(\la^-_0)^2 \frac{f'(P)}{f'(0)} - \left[ \frac{M_1-c}{D}\,\la^-_0  - \frac{f'(0)}{D}\right]
- \frac{M_2-M_1}{D} \,\la^-_0 \, p    \le 0\, ,
$$
But by \eq{quadratic} the terms in the brackets are equal to $(\la^-_0)^2$ and thus \eq{outcond} can be written as 

\begin{equation}
(\la^-_0)^2 \left( \frac{f'(P)}{f'(0)} - 1\right) + \frac{M_1-M_2}{D} \,\la^-_0 \, p \le 0\, .
\label{aux3}
\end{equation}
Note that $\la^-_0 < 0$ and $f(P)$ is given by \eq{Fquadratic}. Therefore \eq{outcond} reduces to 
\begin{equation}
\la^-_0 \le \frac{M_1 - M_2}{2D}\, .
\label{laestim}
\end{equation}
As it was shown above $d \la^-_0(c) / dc < 0$ for all $c \ge M_1 + 2\sqrt{f'(0)D}$ and thus 
if $M_1 - M_2 \ge 2\sqrt{f'(0)D}$ then \eq{laestim} holds for all $c \ge M_1 + 2\sqrt{f'(0)D}$. 
On the other hand, if $M_1 - M_2 < 2\sqrt{f'(0)D}$ then \eq{laestim} holds for all $c$ bigger or equal than the only root of 
the equality in \eq{laestim}. A simple calculation gives that it is indeed $c \ge \hat{c}$. 

Note that the condition \eq{outcond} locally around $(1,0)$ also guarantees that $W^u_1(c)$ lies inside $\mathcal R$. One may check this claim also by a direct calculation. The condition can be written as 
\begin{equation}
\la^+_1(c) + \la^-_0(c) \le 0\, .
\label{eigencond}
\end{equation} 
By plotting the curves 
$D\la^+_1(c)$ and $-D\la^-_0(c)$ in two different cases we deduce that \eq{eigencond} is equivalent to $c \ge c^{\ast}$ for  $M_2 \ge M_1 + 2\sqrt{f'(0)D}$ and to $c \ge M_1 + 2\sqrt{f'(0)D}$ for  $M_2 < M_1 + 2\sqrt{f'(0)D}$.

It remains to prove that for $c = \hat{c}$ the function satisfying \eq{sol} solves \eq{secondorder}. But that follows immediately from the fact that for $\la = \hat{\la}$ there is an equality in \eq{laestim}.  
\end{proof}

\section{Quadratic Nonlinearity with Non-Ho\-mo\-ge\-ne\-ous Diffusion}
\label{s:nonhom}
We have performed numerical calculation of the critical speed $c^{\ast}$ in the case the diffusion coefficients $D_1$ and $D_2$ do not agree and monostable $f(P)$ satisfying \eq{Fquadratic}. Our results are illustrated on Fig.~\ref{fig:quadratic}. We observe two important features in the behavior of $c^{\ast}$ as a function of $M_2$. 

\begin{figure}[t]
\centering
\includegraphics[width=4.5in]{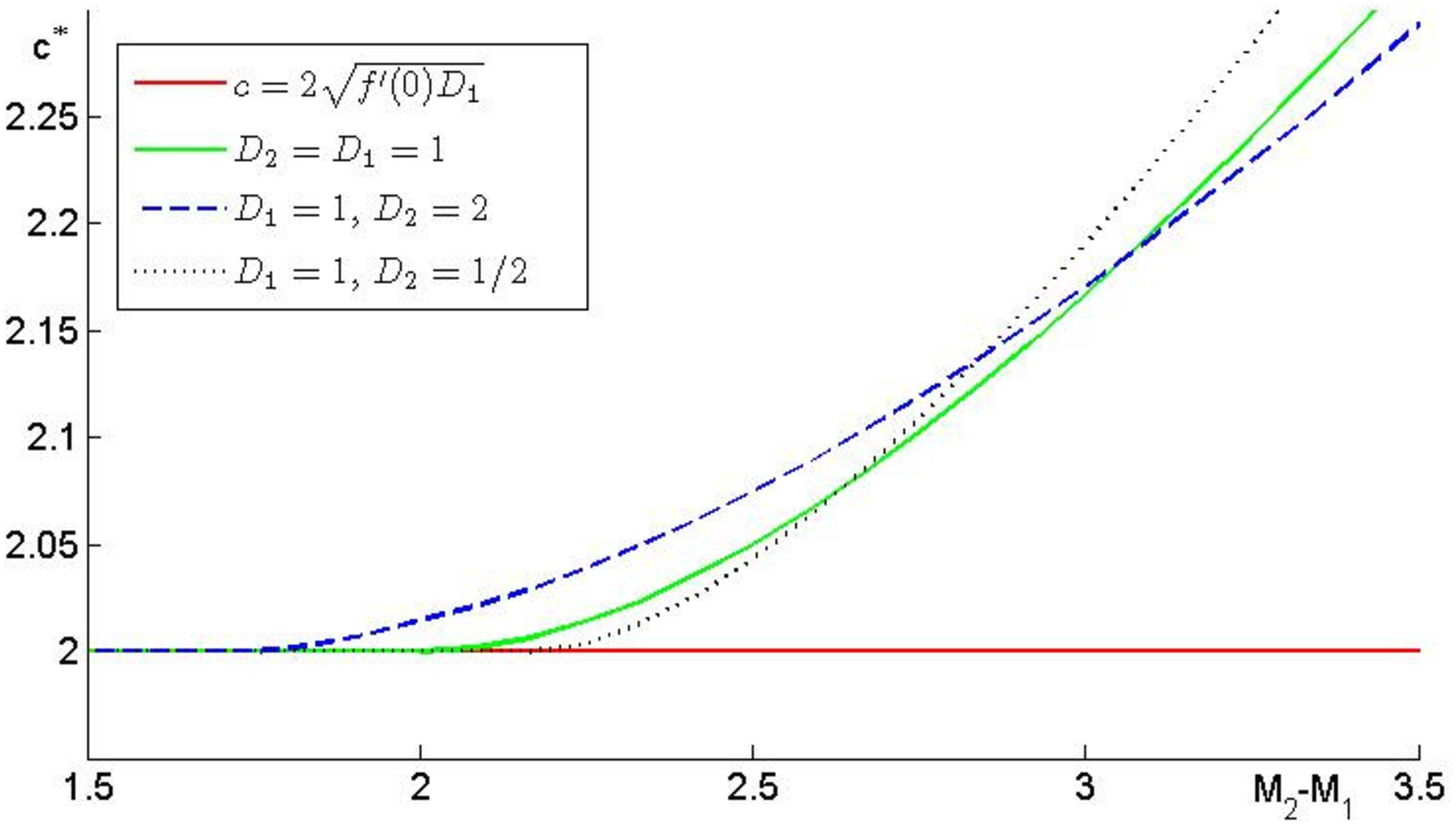}
\caption{Values of the minimum critical wave speed $c^{\ast}$ for $f(P) = P(1-P)$ and $D_1 = 1$. The dependence of $c^{\ast}$ on $M_2 - M_1$ is shown for $D_2 = 0.5$ (dotted line), $D_2 =  1$ (solid line), and $D_2 = 2$ (dashed line). The horizontal line $c^{\ast} = 2$ corresponds to the speed of the pulled wave $c_{pull}=2\sqrt{f'(0)D_1}$.}
        \label{fig:quadratic}
\end{figure}

First, on Fig.~\ref{fig:quadratic} one may notice that for $D_1 \neq D_2$ the transition value of $M_2 - M_1$ at which the minimum traveling wave speed $c^{\ast}$ changes from the pulled wave speed $c^{\ast} = c_{pull} = M_1 + 2\sqrt{D_1f'(0)}$ to a pushed wave speed $c^{\ast} > c_{pull}$. Figure~\ref{fig:quadratic} suggests that  the transition value of $M_2 - M_1$ at which the pulled wave stops to exists decreases for $D_1 < D_2$ and increases for $D_1 > D_2$. However, this is not completely true as can be seen on Fig.~\ref{fig:transition}.

We observe that for moderate values of $D_2/D_1$ the change is approximately linearly depending on $D_2$ (for fixed value of $D_1$) with the negative slope approximately $-0.224$, i.e. 
$$
(M_2-M_1)_{trans}  = -0.224(D_2 - D_1) + 2\sqrt{D_1f'(0)}\, .
$$
However, in the case of bigger mismatch between the diffusion coefficients, $D_2/D_1 > 3$ and $D_2/D_1 < 0.3$, the linear approximation is no longer valid. Particularly, for $D_2 > 3.7$ (approximately) the transition point start to move to higher values of $M_2 - M_1$, and for $D_2/D_1 > 8$ (approximately) it reaches values bigger than its value for $D_2 = D_1$. 
Furthermore,  the graph of $c^{\ast}$ as a function of $M_2-M_1$ also changes its shape; while for moderate values of $D_2/D_1$ the graph is concave up (see Fig.~\ref{fig:quadratic}), for $D_2/D_1 \gg 1$ it becomes concave down. The transition value $(M_2 - M_1)_{trans}$ of $M_2-M_1$ at which transition from pulled waves to pushed waves occurs for large $D_2/D_1$ depends approximately linearly on $D_2$ with the positive slope approximately 0.2791. Similarly, for  $D_2 < 0.12$ (approximately) the transition point moves to smaller values of $M_2 - M_1$, see the inset on Fig.~\ref{fig:transition}. 

While the explanation of the approximately behavior for the moderate values of $D_2/D_1$ is not surprising, as it is caused by the fact that if $D_2 > D_1$ the effect of fast advection $M_2$ is transported more efficiently from the bulk of the wave to its tail and thus even drifts $M_2  < c_{pull}$ can cause a speed up of the wave. On the other hand, if $D_1 > D_2$ the diffusion of the bulk phase is smaller and thus the influence of its fast drift on the tail is weaker. Therefore, the wave can travel with the pull wave speed even if $M_2  > c_{pull}$, although once  $M_2 - M_1$ passes a certain transition threshold $(M_2 - M_1)_{trans}$, the pulled wave is not admissible. 

On the other hand, the behavior of the dependence of the transition value of $(M_2 - M_1)_{trans}$ on $D_2/D_1$ for large and small values of $D_2/D_1$ is unclear. A large mismatch in the diffusion coefficients leads to a strong nonlinear effect. Note that in a different context a strong effect of a mismatch in the diffusion coefficients in a system of couple reaction-diffusion equations is known to be responsible for pattern formation.

\begin{figure}[t]
\centering
\includegraphics[width=0.8\textwidth]{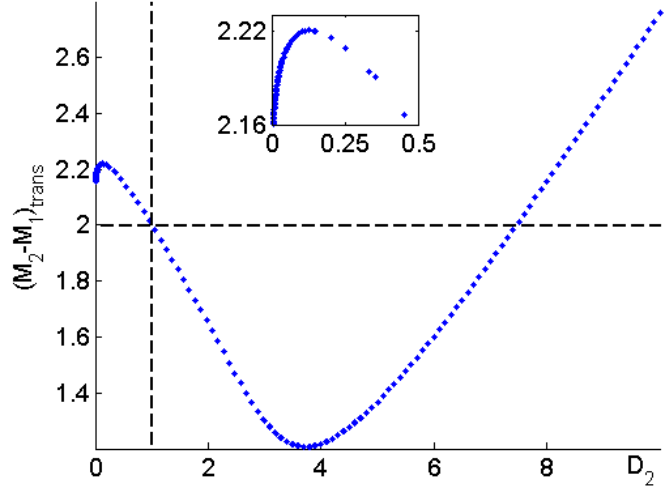}
\caption{The value of $(M_2 - M_1)_{trans}$ at which the minimal speed of a traveling wave for $f(P) = P(1-P)$ transitions from the pulled wave regime, $c^{\ast} = c_{pull}$, to a pushed wave regime, $c^{\ast} > c_{pull}$, as a function of $D_2$. Here $D_1 = 1$, i.e., $c_{pull} = 2$. A detail graph for small values of $D_2$ is plotted on the inset.}
        \label{fig:transition}
\end{figure}

The other interesting feature is the asymptotic behavior of the minimal critical speed of the wave $c^{\ast}$ as $M_2 - M_1 \rightarrow \infty$.The approximately linear behavior of $c^{\ast}$ for $M_2 - M_1 \rightarrow \infty$ can be seen on Fig.~\ref{fig:quadratic}.  This is certainly true for $D_1 = D_2$ as in Section~\ref{s:homogeneous} we have proved that $c^{\ast} \propto (M_2 - M_1)/2$. Our numerical results shown on Fig.~\ref{fig:asymp} demonstrate that $c^{\ast}$ grows faster as a function $M_2 - M_1$ if $D_2 < D_1$ and slower if $D_2 > D_1$. This observation is also in agreement with the expectation that for fixed large value of $M_2 - M_1$ a stronger diffusion of the bulk ($D_2 > D_1$) will allow waves with lower speeds (closer to $c_{pull}$). The effect of the weaker diffusion is the opposite. Also note (see the inset on Fig.~\ref{fig:asymp}) that at least for moderate values of $D_2/D_1$ the linear factor $K$ in $c^{\ast} \propto K(M_2 - M_1)$ has a logarithmic correction factor to the value $1/2$ reached at $D_2/D_1 = 1$.  

\begin{figure}[t]
\centering
\includegraphics[width=0.8\textwidth]{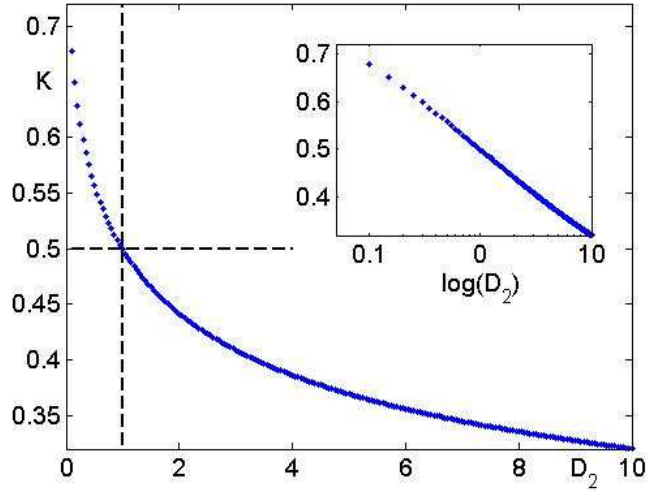}
\caption{The asymptotic slope $K $ of the minimal critical wave speed $c^{\ast} \propto K (M_2 - M_1)$ as $M_2 - M_1 \rightarrow \infty$ for $f(P) = P(1-P)$ and $D_1 = 1$ (evaluated at $M_2 - M_1 = 50$) as a function of $D_2$.  The approximately linear dependence on $\log(D_2)$ is shown on the inset.}
\label{fig:asymp}
\end{figure}

\section{Discussion}
\label{s:conclusion}
We have characterized the parameter regimes for which the traveling waves for \eq{FKPP} exist. In the special case $D_1 = D_2$ and $f(p) = kp(1-p)$ we were able to exactly determine the range of admissible wave speeds and for $D_1 \neq D_2$ we numerically analyzed the critical (minimal) wave speed. 

Our results have implications for applications as type-dependent dispersal is, in general, a prerequisite for studying the evolution of dispersal strategies themselves. Moreover, one may disregard the evolution of dispersal strategies and study type-dependent dispersal in any given biological system, e.g., its effect on spatial gene frequency patterns, see \cite{NK2016} for summary of our results in this direction. Assuming no difference in local growth rates ($f(p)=0$), \cite{Cantrell2008} used the reaction-diffusion framework to identify a class of dispersal strategies that is selectively superior to other dispersal strategies.

With (monostable) quadratic $f(p)$, the FKPP equation describes the spread of a beneficial gene through a population in the form of a traveling wave \citep{Fisher}.Using the generalized FKPP equation \eq{FKPP} we find that type-dependent dispersal may accelerate the wave, yet not delay or reverse its advance.
If $f(p)$ is cubic (bistable), the wave profile may describe gene frequency transitions between species in hybrid zones \citep{Barton1979}. Considering the action of type-dependent dispersal on the speed and width of such waves helps to refine the conclusions about active selection intensities and the timing of secondary contact between species. Also, estimating the precise form of the function $f$ is difficult in practice. Hence, it is valuable to derive conceptual statements as we did in the analysis at hand and our analysis of the speed of the traveling wave may help to identify the unknown biological parameters. 

Our analysis of the generalized FKPP equation \eq{FKPP} also brings a number of open problems. 
To date, a deeper mathematical understanding of the generalized FKPP equation and its traveling wave solutions is just developing. A particularly biologically relevant challenge will be to extend the equation to systems of more than two genotypes.
In the context of evolutionary game theory, traveling waves with three types being present in the population have been studied by \cite{Hutson2002}. However, a general understanding of a multi-dimensional version of the system is still missing.

Another important question to answer is whether the traveling waves which existence we proved are  stable, i.e.,
whether the wave that is initially perturbed within a certain class of admissible perturbations will asymptotically approach its unperturbed form. It is easy to see that similarly as for \eq{tFKPP}, the traveling waves are unstable as they do not need to converge to its exact form but rather to its spatial shifts that encode the extent of the perturbation. This is clearly demonstrated both by the presence of zero in the spectrum of the problem linearized around the traveling wave (that corresponds to the invariance of the dynamics with respect to the spatial shift) and by the presence of the continuous spectrum of the same linearized problem on the real line containing zero it is interior.  Since these results for \eq{tFKPP} only depend on the asymptotic behavior of the system close to $\xi \rightarrow \pm \infty$, they are identical for \eq{FKPP} as the system is approximately constant in the asymptotic regime. 
On the other hand, for \eq{tFKPP} it is possible to consider the stability problem in the exponentially weighted space that moves the continuous spectrum into the left complex half-plane. That means that the traveling waves are indeed (orbitally) stable with respect to infinitesimal perturbations that decay sufficiently fast. Furthermore, the spectral stability in combination with resolvent estimates in the appropriate functional spaces can be used to prove the nonlinear stability with respect to small enough perturbations of the same class. 

But the nonlinearity in the leading order term of \eq{FKPP} introduces severe technical difficulties that make the techniques used in the proofs of linear and nonlinear stability of traveling waves for \eq{tFKPP} hard to extend. First, the natural exponential weight involves the traveling wave profile itself. That may be overcome by restricting the perturbations to a smaller functional space, however, such a step may be too restrictive. Second, the resolvent estimates used in the proof of the nonlinear stability of the traveling waves for \eq{tFKPP} are not sufficient to establish stability for \eq{FKPP}, particularly, it is not clear how one can control nonlinear terms that involve the second derivative of $p$. Thus the extension of the existing theory to the nonlinear setting is not straightforward and the technical difficulties stemming from the (weak) nonlinearity in the diffusion term require an alternative approach. Any results in this direction can be of general interest for various problems of similar type. 

Furthermore, despite the fact that we conjecture it is not possible to determine explicitly the minimum critical speed $c^{\ast}$ discussed in Section~\ref{s:nonhom} in the case $D_1 \neq D_2$, asymptotic analysis may reveal the dependence of the asymptotic slope of the curve $c^{\ast}(M_2-M_1)$ on $D_2/D_1$, and also the dependence of the transition point $(M_2 - M_1)_{trans}$ at which $c^{\ast}$ becomes bigger than $c_{pull}$ on the same parameter $D_2/D_1$, at least in some parameter regimes. Particularly interesting would be to rigorously explain the nonlinear dependence of the transition value of $M_2-M_1$ on $D_2/D_1$ on Fig.~\ref{fig:transition}. 

It would also be interesting to extend some of the recent results obtained by \cite{Polacik2016} on the global dynamics for the Cauchy problem for \eq{tFKPP} to \eq{FKPP}, or to remove the assumption on non-degeneracy at equilibria $f'(0) \neq 0$, $f'(1) \neq 0$. Similarly, one can try to determine how the nonlinear diffusion or drift influence the analysis in \cite{BD1997, DKP2007} and in \cite{DK2015} in the asymptotic regime in which $f(p)$ is modified in the $\eps$-neighborhood of $p = 0$. 

On the other hand, some questions remain unanswered in the process of derivation of \eq{FKPP} from the system \eq{eqn1}. The main problem is to determine for what classes of initial conditions and parameter values does the system  \eq{redeq} coupled with  \eq{totaleq} approximate the system \eq{fullexp} and \eq{totaleq} well and to quantify the speed of the growth of the deviation in time. A simpler toy  problem is to consider the coupled system \eq{fullexpred}--\eq{totaleqpred} and compare its dynamics with \eq{fullexp} and \eq{totaleq}.

\section*{Acknowledgment}
We thank Nick Barton, Katar{\'i}na Bo{\md}ov{\'a}, and Srdjan Sarikas for constructive feedback and support. 
This project has received funding from the European Union's Seventh Framework Programme for research, technological development and demonstration under Grant Agreement 618091 Speed of Adaptation in Population Genetics and Evolutionary Computation (SAGE) and the European Research Council (ERC) grant no. 250152 (SN), from the Scientific Grant Agency of the Slovak Republic under the grant 1/0459/13 and by the Slovak Research and Development Agency under the contract No.~APVV-14-0378 (RK). RK would also like to thank IST Austria for its hospitality during the work on this project. 


\end{document}